\def\bu{\bullet}
\def\marker{\>\hbox{${\vcenter{\vbox{
    \hrule height 0.4pt\hbox{\vrule width 0.4pt height 6pt
    \kern6pt\vrule width 0.4pt}\hrule height 0.4pt}}}$}\>}
\def\gpic#1{#1
     \smallskip\par\noindent{\centerline{\box\graph}} \medskip}

\documentclass[12pt]{article}
\usepackage{array,amsmath,amssymb,amsthm}  				    
\usepackage{fullpage,lineno}

\usepackage{color}

\begin{document}

\newtheorem{theorem}{Theorem}[section]
\newtheorem{lemma}[theorem]{Lemma}
\newtheorem{corollary}[theorem]{Corollary}
\newtheorem{prop}[theorem]{Proposition}
\newtheorem{conj}[theorem]{Conjecture}
\newtheorem{claim}[theorem]{Claim}
\theoremstyle{definition}
\newtheorem{defn}[theorem]{Definition}
\newtheorem{remark}[theorem]{Remark}
\newtheorem{question}[theorem]{Question}
\newtheorem{construction}[theorem]{Construction}
\newtheorem{alg}[theorem]{Algorithm}
\def\qed{\ifhmode\unskip\nobreak\hfill$\Box$\bigskip\fi \ifmmode\eqno{Box}\fi}

\def\nul{\varnothing} 
\def\st{\colon\,}   
\def\e{{\rm e}}
\def\esub{\subseteq}
\def\VEC#1#2#3{#1_{#2},\ldots,#1_{#3}}
\def\VECOP#1#2#3#4{#1_{#2}#4\cdots #4 #1_{#3}}
\def\SE#1#2#3{\sum_{#1=#2}^{#3}} 
\def\PE#1#2#3{\prod_{#1=#2}^{#3}}
\def\UE#1#2#3{\bigcup_{#1=#2}^{#3}}
\def\UM#1#2{\bigcup_{#1\in #2}}
\def\CH#1#2{\binom{#1}{#2}} 
\def\FR#1#2{\frac{#1}{#2}}
\def\FL#1{\left\lfloor{#1}\right\rfloor} \def\FFR#1#2{\FL{\frac{#1}{#2}}}
\def\CL#1{\left\lceil{#1}\right\rceil}   \def\CFR#1#2{\CL{\frac{#1}{#2}}}
\def\Gb{\overline{G}}
\def\NN{{\mathbb N}} \def\ZZ{{\mathbb Z}} \def\QQ{{\mathbb Q}}
\def\RR{{\mathbb R}} \def\GG{{\mathbb G}} \def\FF{{\mathbb F}}
\long\def\skipit#1{{}}

\def\B#1{{\bf #1}}      \def\R#1{{\rm #1}}
\def\I#1{{\it #1}}      \def\c#1{{\cal #1}}
\def\C#1{\left | #1 \right |}    
\def\P#1{\left ( #1 \right )}    
\def\ov#1{\overline{#1}}        \def\un#1{\underline{#1}}

\expandafter\ifx\csname dplus\endcsname\relax \csname newbox\endcsname\dplus\fi
\expandafter\ifx\csname dplustemp\endcsname\relax
\csname newdimen\endcsname\dplustemp\fi
\setbox\dplus=\vtop{\vskip -8pt\hbox{%
    \kern .2em
    \special{pn 6}%
    \special{pa 0 50}%
    \special{pa 50 100}%
    \special{fp}%
    \special{pa 50 100}%
    \special{pa 100 50}%
    \special{fp}%
    \special{pa 100 50}%
    \special{pa 50 0}%
    \special{fp}%
    \special{pa 50 100}%
    \special{pa 50 0}%
    \special{fp}%
    \special{pa 50 0}%
    \special{pa 0 50}%
    \special{fp}%
    \special{pa 0 50}%
    \special{pa 100 50}%
    \special{fp}%
    \hbox{\vrule depth0.080in width0pt height 0pt}%
    \kern .7em
  }%
}%
\def\gjoin{\copy\dplus}

\def\cG{{\mathcal G}}
\def\cH{{\mathcal H}}
\def\la{\langle}
\def\ra{\rangle}
\def\Imp{\Rightarrow}
\def\CQ{\mathcal{Q}}
\def\CA{\mathcal{A}}
\def\hG{\hat G}


\title{Sharp lower bounds for the number of maximum matchings in bipartite
multigraphs}

\author{ 
Alexandr V. Kostochka\thanks{University of Illinois, Urbana, IL,
and Sobolev Institute of Mathematics, Novosibirsk, Russia:
\texttt{kostochk@illinois.edu}.  Research supported by NSF
grants DMS-1937241 (RTG) and DMS-2153507.}\,,
Douglas B. West\thanks{Zhejiang Normal University, Jinhua, China,
and University of Illinois, Urbana, IL: \texttt{dwest@illinois.edu}.
Research supported by National Natural Science
Foundation of China grants NSFC 11871439, 11971439, and U20A2068.}\,,
Zimu Xiang\thanks{University of Illinois, Urbana, IL:
\texttt{zimux2@illinois.edu}.}
}
\date{\today}
\maketitle

\baselineskip 16pt

\begin{abstract}
We study the minimum number of maximum matchings in a bipartite multigraph $G$
with parts $X$ and $Y$ under various conditions, refining the well-known lower
bound due to M.~Hall.  When $\C X=n$, every vertex in $X$ has degree at least
$k$, and every vertex in $X$ has at least $r$ distinct neighbors, the minimum
is $r!(k-r+1)$ when $n\ge r$ and is $[r+n(k-r)]\PE i1{n-1}(r-i)$ when $n<r$.
When every vertex has at least two neighbors and $\C Y-\C X=t\ge0$,
the minimum is $[(n-1)t+2+b](t+1)$, where $b=\C{E(G)}-2(n+t)$.
We also determine the minimum number of maximum matchings in several
other situations.  We provide a variety of sharpness constructions.
\end{abstract}

\section{Introduction}
We study lower bounds for the number of maximum matchings in bipartite
multigraphs.  A {\it matching} is a set of pairwise disjoint edges.
An {\it $X,Y$-bigraph} is a bipartite (multi)graph with parts $X$ and $Y$.
For a set $S$ of vertices in a graph, let $N(S)$ denote the union of the
neighborhoods of the vertices in $S$.  {\it Hall's Condition} for $X$ in an
$X,Y$-bigraph is the condition that $\C{N(S)}\ge\C S$ for every subset $S$ of
$X$.  The celebrated theorem of Philip Hall~\cite{PHa} for an $X,Y$-bigraph $G$
states that if Hall's Condition holds for $X$, then $G$ contains a matching
covering the vertices of $X$.  We call such a matching an {\it $X$-matching.}
Let $\Phi(G)$ denote the number of maximum matchings (matchings with the most
edges) in $G$.  When Hall's Condition holds, the maximum matchings are
$X$-matchings.

We use the term {\it simple $X,Y$-bigraph} when multiedges are forbidden.
For a set $S$ of vertices in a graph $G$, let $\delta_S(G)=\min_{v\in S}d_G(v)$,
where $d_G(v)$ denotes the degree of vertex $v$ in $G$.
Marshall Hall~\cite{MHa} proved $\Phi(G)\ge \PE i1{\min\{k,n\}}(k-i+1)$ when
$G$ is a simple $X,Y$-bigraph with $\C{X}=n$ and $\delta_X(G)\ge k$ that
satisfies Hall's Condition for $X$.  In particular, when $k\le n$ this lower
bound is $k!$, and its sharpness for simple graphs is a special case of our
general sharpness construction.

There are also lower bounds for bipartite multigraphs.  A multigraph is
{\it $k$-regular} if every vertex has degree $k$.  The famous theorem of
Egorychev~\cite{Ego} and Falikman~\cite{Fal}, proving what was known as
van der Waerden's Conjecture, implies that every $k$-regular bipartite
multigraph with $n$ vertices in each part has at least $n!(k/n)^n$ perfect
matchings (that is, matchings covering all the vertices).  Since
$n!\approx (n/\e)^n\sqrt{2\pi n}$ by Stirling's Approximation, the leading
order of growth in this lower bound for $k$-regular multigraphs is $(k/\e)^n$,
while Hall's lower bound for simple graphs, which is fixed for large $n$, is
only about $(k/\e)^k$.

In order to combine these situations, we enlarge the graph context by allowing
multiedges and the multigraph context by weakening regularity to a minimum
degree requirement applied only to $X$.  To further obtain a spectrum of
problems that in some sense bridges the gap between bipartite graphs and
bipartite multigraphs, we introduce a neighborhood condition and prove the
following result.

\begin{theorem}\label{main}
Fix $k,r,n\in\NN$ with $k\ge r$.  Let $G$ be an $X,Y$-bigraph with $\C X=n$.
If $\delta_X(G)\ge k$ and $\C{N_G(x)}\ge r$ for all $x\in X$, then
$$
\Phi(G)\ge\begin{cases}
   r!(k-r+1) &\text{if $n\ge r$,}\\
   [r+n(k-r)]\PE i1{n-1}(r-i)&\text{if $n\le r$.}
\end{cases}
$$
Furthermore, the bounds are sharp in all cases.
\end{theorem}

When $r=k$, the bound simplifies to $\PE i0{\min\{n,k\}-1}(k-i)$, the bound of
M. Hall.  This is achieved in our Construction~\ref{sharp1}.  The proof will
show also that $\Phi(G)$ attains this minimum only when $G$ satisfies Hall's
Condition and the maximum matchings are $X$-matchings.  In Theorem~\ref{maineq}
we further explore the characterization of constructions achieving equality in
the lower bound, in which Construction~\ref{sharp1} plays a crucial role.

Setting $r=1$ in Theorem~\ref{main} is equivalent to eliminating the
restriction on $r$.  Theorem~\ref{main} then only guarantees $\Phi(G)\ge k$ in
an $X,Y$-bigraph $G$ with $\delta_X(G)\ge k$ satisfying Hall's Condition, which
is sharp by the construction in Construction~\ref{sharp1}, independent of $n$.
However, in this construction $\C Y=\C X$ and $\delta_Y(G)=1$.  Forbidding this
situation yields a better lower bound, further indicating the sharpness of the
previous result.

\begin{theorem}\label{Y2}
Let $G$ be an $X,Y$-bigraph with $\delta_X(G)\ge k$ and $\delta_Y(G)\ge1$
that satisfies Hall's Condition.  If $\C Y>\C X$ or $\delta_Y(G)\ge2$, then
sharp lower bounds on $\Phi(G)$ are as follows:
$$
\Phi(G)\ge\begin{cases}
   2k-2, &\text{if $\C X=2$ or $k=2$;}\\
   2k-1, &\text{if $\C X>2$ and $k=3$;}\\
   2k,   &\text{if $\C X\ge3$ and $k\ge4$.}
\end{cases}
$$
\end{theorem}

We will see that the extremal configurations in Theorem~\ref{Y2} can also be
excluded by a further restriction, partially combining the restrictions of
Theorems~\ref{main} and~\ref{Y2}.

\begin{theorem}\label{4k}
For $k,n\geq 2$, let $G$ be an $X,Y$-bigraph with $\delta_X(G)\geq k$ and
$\delta_Y(G)\geq 2$ that satisfies Hall's Condition.  If $|N_G(x)|\geq 2$ for
each $x\in X$, then
\begin{equation*}\label{e1}
\Phi(G)\geq \min\{n(k-2)+2, 4k-4\}.
\end{equation*}
\end{theorem}

In these two results, the guaranteed lower bound on $\Phi(G)$ does not grow as
$\C Y$ grows, but the sharpness examples generally require $\C Y=\C X$.  In
general, as $\C Y-\C X$ grows, the guaranteed number of $X$-matchings also
grows.  Some of the results in the next theorem reduce to parts of
Theorem~\ref{Y2} when $t=0$ and need Theorem~\ref{Y2} as a basis.

\begin{theorem}\label{2kt}
Let $G$ be an $X,Y$-bigraph with $\C X\ge2$, $\delta_X(G)\ge k$, and
$t=\C Y-\C X\ge0$.  If Hall's Condition holds for $X$, then
$$
\Phi(G)\ge\begin{cases}
   k(t+1), &\text{if $\delta_Y(G)\ge1$;}\\
   (2k-t-2)(t+1), &\text{if $\delta_Y(G)\ge2$ and $\C X=2$;}\\
   2k(t+1)-1, &\text{if $\delta_Y(G)\ge2$ and $k=3$;}\\
   2k(t+1), &\text{if $\delta_Y(G)\ge 2$ and $k\ge4$.}
\end{cases}
$$
\end{theorem}

Finally, when every vertex in $G$ must have at least two neighbors and subsets
of $X$ must have neighborhoods larger than their size, the lower bound on
$\Phi(G)$ in terms of the various parameters is surprisingly large, and it
grows with $\C X$, with $\C Y-\C X$, and with $\C{E(G)}$.

\begin{theorem}\label{leafmain}
Let $G$ be an $X,Y$-bigraph in which every vertex has at least two neighbors.
Let $n=\C X\ge2$, and $t=\C Y-\C X\ge0$, and $b=\C{E(G)}-2\C Y$.
If $\C{N(S)}>\C S$ for every nonempty proper subset $S$ of $X$, then
$$\Phi(G)\ge [(n-1)t+2+b](t+1),$$
and the lower bound is sharp in all cases.
\end{theorem}

In Section 2, we prove Theorem~\ref{main} and discuss the requirements for
equality.  Section 3 develops tools applicable to special cases of
Theorem~\ref{4k} and Theorem~\ref{leafmain}.  Section 4 contains the proof of
Theorem~\ref{leafmain}, our most difficult result, which also has applications
in the later sections.  In Section 5, we consider lower bounds when the
neighborhood requirement for vertices in $X$ is replaced by a condition on $Y$,
as in Theorem~\ref{Y2}.  Section 6 considers the presence of excess vertices
in $Y$, as in Theorem~\ref{2kt} and Theorem~\ref{leafmain}.  Sharpness
constructions are given for each lower bound.

Many of our proofs follow the technique of P.~Hall~\cite{PHa}, Halmos and
Vaughan~\cite{HV}, and Mann and Ryser~\cite{MR}, which separates the problem
into two cases depending on whether some proper subset $S$ of $X$ satisfies
$\C{N(S)}=\C S$.  This distinction is related to a classical notion studied by
Lov\'asz and Plummer~\cite{LP}.  In a graph with a perfect matching, an edge is
{\it allowed} if it belongs to a perfect matching, otherwise {\it forbidden}.
Hetyei~\cite{He} defined a graph to be {\it elementary} if its allowed edges
form a connected subgraph.  It is easy to see that an $X,Y$-bigraph with a
perfect matching is elementary if and only if $\C{N(S)}>\C S$ for every
nonempty proper subset $S$ of $X$.  In such graphs, all edges are ``allowed''.

In some sense, the case $\C{N(S)}=\C S$ reduces the problem to a smaller
graph, and more interesting behavior arises when this is forbidden.
In order to study lower bounds on $\Phi(G)$ in a simple $X,Y$-bigraph $G$,
Liu and Liu~\cite{LL} introduced the concept of {\it positive surplus} for
an $X,Y$-bigraph, meaning $\C{N(S)}>\C S$ for all nonempty $S$ in $X$, 
including $S=X$.  In particular, $\C Y > \C X$ for such a graph.
We say that an $X,Y$-bigraph is {\it $X$-surplus} if $\C{N(S)}>\C S$ for every
nonempty proper subset $S$ of $X$; this condition is a common extension of
``positive surplus'' and ``elementary bipartite''.  In particular,
Theorem~\ref{leafmain} applies to all $X$-surplus $X,Y$-bigraphs
in which every vertex of $Y$ has at least two neighbors.

Liu and Liu proved several lower bounds for the number of maximum matchings
in a simple $X,Y$-bigraph with positive surplus, which we combine into one
statement.  Note that maximum matchings in such graphs are $X$-matchings.

\begin{theorem}[{\rm Liu and Liu~\cite{LL}}]\label{ll}
If $G$ is a simple $X,Y$-bigraph with positive surplus, and $\C Y -\C X=t\ge1$,
then 
$$
\Phi(G)\ge\begin{cases}
   \C X+1, &\text{if $G$ is connected;}\\
   \C{E(G)}+(\C X-1)(t-2), &\text{if $G$ is connected;}\\
   2\C{E(G)}-2\C Y, &\text{if $t=1$ and $\delta(G)\ge2$;}
\end{cases}
$$
\end{theorem}

The lower bound $2\C{E(G)}-2\C Y$ when $t=1$ and $\delta(G)\ge2$ is the
special case $t=1$ of our general Theorem~\ref{leafmain} (restricted to
simple graphs), since $[(n-1)1+2+b]2=2\C{E(G)}-2\C Y$.  Our proof of the
general result includes a proof of their result for $t=1$.  The case
$t=0$ we prove by different methods in Section~\ref{sec:surp}, and it serves as
a base case for Theorem~\ref{leafmain}.  It can be stated as follows, allowing
multiedges.

\begin{theorem}\label{surplus}
If $G$ is an elementary $X,Y$-bigraph, then $\Phi(G)\ge \C{E(G)}-\C{V(G)}+2$,
and this is sharp.
\end{theorem}
We prove Theorem~\ref{surplus} as a consequence of the existence of an ear
decomposition of an elementary $X,Y$-bigraph in which every added ear beyond
the initial even cycle has odd length (obtained in~\cite{He,Lov,LP}).
A further consequence of this material in Section~\ref{sec:surp} will be
Theorem~\ref{4k} for the case $\C X=\C Y$.

\section{Degree and Neighborhood Restrictions on $X$}

In this section we consider $X,Y$-bigraphs where every vertex of $X$ has
at least $r$ neighbors and has degree at least $k$.  Thus $r\le k$.
When $r$ is smaller than $k$, the lack of a degree condition on $Y$ allows
many edges of the multigraph to be incident to one vertex of $Y$, leaving
little flexibility for the matching.  Our sharpness construction exploits this.
Let $\cG_{n,k,r}$ be the family of $X,Y$-bigraphs $G$ satisfying Hall's
Condition such that $\C X=n$, $\delta_X(G)\ge k$, and $\C{N_G(x)}\ge r$ for all
$x\in X$.  

\begin{construction}\label{sharp1}
{\it A graph $G\in\cG_{n,k,r}$ with $\Phi(G)= r!(k-r+1)$ when $n\ge r$,
or $\Phi(G)=[r+n(k-r)]\PE i1{n-1}(r-i)$ when $n\le r$.}
Let $X=\{\VEC x1n\}$ and $Y=\{\VEC y1m\}$, where $m={\max\{r,n\}}$.  Begin with
a copy of $K_{n,r}$ having parts $X$ and $\{\VEC y1r\}$.  Add $k-r$ copies of
the edge $x_iy_1$ for $1\le i\le n$.  Finally, in the case $n>r$, add the edge
$x_iy_i$ for $r<i\le n$.  Let $G$ be the resulting $X,Y$-bigraph; every vertex
in $X$ has degree at least $k$ and has at least $r$ neighbors.
See Figure~\ref{mainfig}.  Maximum matchings are $X$-matchings and have
size $n$.

\begin{centering}
\begin{figure}[h]
\gpic{
\expandafter\ifx\csname graph\endcsname\relax \csname newbox\endcsname\graph\fi
\expandafter\ifx\csname graphtemp\endcsname\relax \csname newdimen\endcsname\graphtemp\fi
\setbox\graph=\vtop{\vskip 0pt\hbox{%
    \graphtemp=.5ex\advance\graphtemp by 0.637in
    \rlap{\kern 0.159in\lower\graphtemp\hbox to 0pt{\hss $\bu$\hss}}%
    \graphtemp=.5ex\advance\graphtemp by 0.372in
    \rlap{\kern 0.159in\lower\graphtemp\hbox to 0pt{\hss $\bu$\hss}}%
    \graphtemp=.5ex\advance\graphtemp by 0.106in
    \rlap{\kern 0.159in\lower\graphtemp\hbox to 0pt{\hss $\bu$\hss}}%
    \graphtemp=.5ex\advance\graphtemp by 0.903in
    \rlap{\kern 0.956in\lower\graphtemp\hbox to 0pt{\hss $\bu$\hss}}%
    \graphtemp=.5ex\advance\graphtemp by 0.637in
    \rlap{\kern 0.956in\lower\graphtemp\hbox to 0pt{\hss $\bu$\hss}}%
    \graphtemp=.5ex\advance\graphtemp by 0.372in
    \rlap{\kern 0.956in\lower\graphtemp\hbox to 0pt{\hss $\bu$\hss}}%
    \graphtemp=.5ex\advance\graphtemp by 0.106in
    \rlap{\kern 0.956in\lower\graphtemp\hbox to 0pt{\hss $\bu$\hss}}%
    \special{pn 8}%
    \special{ar -890 -1800 2655 2655 0.801525 1.164062}%
    \special{ar 2006 2544 2655 2655 -2.340068 -1.977530}%
    \special{ar -57 -1607 1991 1991 1.036634 1.461458}%
    \special{ar 1173 2085 1991 1991 -2.104959 -1.680135}%
    \special{ar 558 -1572 1726 1726 1.337928 1.803665}%
    \special{ar 558 1785 1726 1726 -1.803665 -1.337928}%
    \special{pn 11}%
    \special{pa 159 637}%
    \special{pa 956 903}%
    \special{fp}%
    \special{pa 159 637}%
    \special{pa 956 637}%
    \special{fp}%
    \special{pa 159 637}%
    \special{pa 956 372}%
    \special{fp}%
    \special{pa 159 372}%
    \special{pa 956 903}%
    \special{fp}%
    \special{pa 159 372}%
    \special{pa 956 637}%
    \special{fp}%
    \special{pa 159 372}%
    \special{pa 956 372}%
    \special{fp}%
    \special{pa 159 106}%
    \special{pa 956 903}%
    \special{fp}%
    \special{pa 159 106}%
    \special{pa 956 637}%
    \special{fp}%
    \special{pa 159 106}%
    \special{pa 956 372}%
    \special{fp}%
    \graphtemp=.5ex\advance\graphtemp by 0.637in
    \rlap{\kern 0.000in\lower\graphtemp\hbox to 0pt{\hss $x_n$\hss}}%
    \graphtemp=.5ex\advance\graphtemp by 0.106in
    \rlap{\kern 0.000in\lower\graphtemp\hbox to 0pt{\hss $x_1$\hss}}%
    \graphtemp=.5ex\advance\graphtemp by 0.903in
    \rlap{\kern 1.142in\lower\graphtemp\hbox to 0pt{\hss $y_m$\hss}}%
    \graphtemp=.5ex\advance\graphtemp by 0.106in
    \rlap{\kern 1.142in\lower\graphtemp\hbox to 0pt{\hss $y_1$\hss}}%
    \graphtemp=.5ex\advance\graphtemp by 0.903in
    \rlap{\kern 2.018in\lower\graphtemp\hbox to 0pt{\hss $\bu$\hss}}%
    \graphtemp=.5ex\advance\graphtemp by 0.637in
    \rlap{\kern 2.018in\lower\graphtemp\hbox to 0pt{\hss $\bu$\hss}}%
    \graphtemp=.5ex\advance\graphtemp by 0.372in
    \rlap{\kern 2.018in\lower\graphtemp\hbox to 0pt{\hss $\bu$\hss}}%
    \graphtemp=.5ex\advance\graphtemp by 0.106in
    \rlap{\kern 2.018in\lower\graphtemp\hbox to 0pt{\hss $\bu$\hss}}%
    \graphtemp=.5ex\advance\graphtemp by 0.903in
    \rlap{\kern 2.814in\lower\graphtemp\hbox to 0pt{\hss $\bu$\hss}}%
    \graphtemp=.5ex\advance\graphtemp by 0.637in
    \rlap{\kern 2.814in\lower\graphtemp\hbox to 0pt{\hss $\bu$\hss}}%
    \graphtemp=.5ex\advance\graphtemp by 0.372in
    \rlap{\kern 2.814in\lower\graphtemp\hbox to 0pt{\hss $\bu$\hss}}%
    \graphtemp=.5ex\advance\graphtemp by 0.106in
    \rlap{\kern 2.814in\lower\graphtemp\hbox to 0pt{\hss $\bu$\hss}}%
    \special{pn 8}%
    \special{ar -371 -2282 3982 3982 0.643501 0.927295}%
    \special{ar 5204 3292 3982 3982 -2.498092 -2.214297}%
    \special{ar 967 -1800 2655 2655 0.801525 1.164062}%
    \special{ar 3864 2544 2655 2655 -2.340068 -1.977530}%
    \special{ar 1800 -1607 1991 1991 1.036634 1.461458}%
    \special{ar 3031 2085 1991 1991 -2.104959 -1.680135}%
    \special{ar 2416 -1572 1726 1726 1.337928 1.803665}%
    \special{ar 2416 1785 1726 1726 -1.803665 -1.337928}%
    \special{pn 11}%
    \special{pa 2814 372}%
    \special{pa 2018 903}%
    \special{fp}%
    \special{pa 2814 372}%
    \special{pa 2018 637}%
    \special{fp}%
    \special{pa 2814 372}%
    \special{pa 2018 372}%
    \special{fp}%
    \special{pa 2814 372}%
    \special{pa 2018 106}%
    \special{fp}%
    \special{pa 2018 903}%
    \special{pa 2814 903}%
    \special{fp}%
    \special{pa 2018 637}%
    \special{pa 2814 637}%
    \special{fp}%
    \graphtemp=.5ex\advance\graphtemp by 0.372in
    \rlap{\kern 3.000in\lower\graphtemp\hbox to 0pt{\hss $y_r$\hss}}%
    \graphtemp=.5ex\advance\graphtemp by 0.903in
    \rlap{\kern 1.858in\lower\graphtemp\hbox to 0pt{\hss $x_n$\hss}}%
    \graphtemp=.5ex\advance\graphtemp by 0.106in
    \rlap{\kern 1.858in\lower\graphtemp\hbox to 0pt{\hss $x_1$\hss}}%
    \graphtemp=.5ex\advance\graphtemp by 0.903in
    \rlap{\kern 3.000in\lower\graphtemp\hbox to 0pt{\hss $y_m$\hss}}%
    \graphtemp=.5ex\advance\graphtemp by 0.106in
    \rlap{\kern 3.000in\lower\graphtemp\hbox to 0pt{\hss $y_1$\hss}}%
    \graphtemp=.5ex\advance\graphtemp by 1.168in
    \rlap{\kern 0.558in\lower\graphtemp\hbox to 0pt{\hss $(k,r,n)=(5,4,3)$\hss}}%
    \graphtemp=.5ex\advance\graphtemp by 1.168in
    \rlap{\kern 2.416in\lower\graphtemp\hbox to 0pt{\hss $(k,r,n)=(3,2,4)$\hss}}%
    \hbox{\vrule depth1.274in width0pt height 0pt}%
    \kern 3.000in
  }%
}%
}

\vspace{-1pc}
\caption{Construction~\ref{sharp1}.  \label{mainfig}}
\end{figure}
\end{centering}

If $n\le r$, then exactly $\PE i1n(r-i)$ $X$-matchings avoid $y_1$.
Those using $y_1$ can choose the edge covering $y_1$ in $n(k-r+1)$ ways, and
then in the remaining copy of $K_{n-1,r-1}$ the rest of $X$ can be covered in
$\PE i1{n-1}(r-i)$ ways.  Hence
$$\Phi(G)=\PE i1n(r-i)+n(k-r+1)\PE i1{n-1}(r-i)=[r+n(k-r)]\PE i1{n-1}(r-i).$$

If $n\ge r$, then $m=n$, the edges $x_iy_i$ for $1\le i\le n$ form an
$X$-matching, and every $X$-matching also covers $Y$.  Since $\VEC y{r+1}n$
have degree $1$, the edges $x_iy_i$ for $r<i\le n$ appear in each $X$-matching.
There are then $r(k-r+1)$ ways to choose the edge covering $y_1$ and $(r-1)!$
ways to match the remaining $r-1$ vertices in $\VEC x1r$ into $\VEC y2n$,
so $\Phi(G)=r!(k-r+1)$.
\end{construction}

Intuitively, removing edges or shrinking neighborhoods should not increase
the number of matchings covering $X$.  Our first lemma makes this precise,
allowing us when $r>1$ to assume that the degree and neighborhood conditions
hold with equality for each vertex in $X$ and that any vertex in $X$ has only
one neighbor via a multiedge.

Just as we use ``$X$-matching'' in an $X,Y$-bigraph, we use
``$X'$-matching'' for a matching of size $\C{X'}$ in an $X',Y$-bigraph.
For $S\esub V(G)$, let $G[S]$ be the subgraph of $G$ induced by $S$.
\looseness-1

Our proof of both the lemma and the lower bound theorem uses the same two basic
cases as in the proofs of P.~Hall's Theorem by M.~Hall~\cite{MHa}, Halmos and
Vaughan~\cite{HV}, and Mann and Ryser~\cite{MR}, depending on whether the
$X,Y$-bigraph is $X$-surplus.

\begin{lemma}\label{restrict}
If $G\in \cG_{n,k,r}$ with $r>1$, then there exists $G'\in\cG_{n,k,r}$ with
$\Phi(G')\le \Phi(G)$ and $V(G')=V(G)$ such that in $G'$ every vertex of $X$
has exactly $r$ distinct neighbors, one with multiplicity $k-r+1$ and the
others with multiplicity $1$.  When $r=1$, in $G'$ we can only guarantee one
such vertex.
\end{lemma}
\begin{proof}
When $r=1$ and the smallest neighborhood size among vertices in $X$ is $1$,
we have $x$ with only one neighbor, and we may delete excess copies of the
incident edge to bring $d_G(x)$ down to $k$ to obtain the desired $G'$.
Hence we may assume $r>1$.

We use induction on $n$.  When $n=1$, we have $\Phi(G)=d_G(x)$, where
$X=\{x\}$.  If $d_G(x)>k$, then we can delete a copy of an edge having highest
multiplicity to reduce the number of $X$-matchings without losing the
hypotheses.  Hence we may assume $d_G(x)=k$, and we can redistribute the edges
so that one pair has multiplicity $k-r+1$ and the other $r-1$ edges have
multiplicity $1$.  Note that $\Phi(G)=k$ for all $G$ in $\cG_{1,k,r}$ having
$k$ edges, so the resulting graph $G'$ minimizing $\Phi$ is not unique.  Now
suppose $n>1$.

\smallskip
{\bf Case 1:} {\it $\C{N(S)}=\C S$ for some nonempty proper subset $S$ of $X$.}
Let $s=\C S$.  Let $G_1=G[S\cup N(S)]$ and $G_2=G-(S\cup N(S)$.  Since
$G\in \cG_{n,k,r}$, there is a perfect matching in $G$; it must consist of
a perfect matching in $G_1$ and an $(X-S)$-matching in $G_2$.

In fact, $\Phi(G)=\Phi(G_1)\Phi(G_2)$.  We first reduce the number of
$X$-matchings by applying the induction hypothesis to $G_1$.  Since $G_1$
retains all edges of $G$ incident to $S$, we have $G_1\in \cG_{s,k,r}$, and
the induction hypothesis provides $G_1'\in \cG_{s,k,r}$ with
$\Phi(G_1')\le\Phi(G_1)$ such that the desired properties holding in $G_1'$ for
all vertices of $S$.  Let $G^o$ be the graph obtained from $G$ by replacing
$G_1$ with $G_1'$ while maintaining all edges incident to $X-S$.

Let $M$ be an $X$-matching in $G^o$, and consider $x\in X-S$.  Let $y$ be the
mate of $x$ in $M$; note that $y\in N(S)-X$.  We will keep all $X$-matchings 
in $G^o$ that use a specific edge $xy$ and destroy all $X$-matchings that do
not use (that copy of) that edge.  In particular, when we do not change the
edges incident to $S$, an $X$-matching cannot use any edge joining $x$ to
$N(S)$.  We replace the edges incident to $x$ by one copy of $xy$ and edges to
$r-1$ vertices of $N(S)$, one with multiplicity $k-r+1$ and the others with
multiplicity $1$.  We can do this because $r-1\ge1$ and $\C{N(S)}\ge r$.
Applying this transformation successively for each vertex of $X-S$ produces
the desired graph $G'$.

\smallskip
{\bf Case 2:} {\it $G$ is $X$-surplus.}
Consider $x\in X$ and $y,y'\in N_G(x)$, which exist since $r>1$.  Let
$X'=X-\{x\}$, and let $\sigma=\Phi(G-\{x,y\})$ and $\sigma'=\Phi(G-\{x,y'\})$.
We may assume $\sigma\le \sigma'$.  Let $G^o$ be the multigraph obtained from
$G$ by (1) shifting all copies of $xy'$ to be copies of $xy$ if $\C{N_G(x)}>r$,
or (2) shifting all but one copy if $\C{N_G(x)}=r$.  Since $G$ is $X$-surplus,
$G^o\in\cG_{n,k,r}$.  Moving $q$ copies of the edge destroys $q\sigma'$
$X$-matchings from $G$ and adds $q\sigma$ new $X$-matchings in $G^o$.  Thus
$\Phi(G^o)-\Phi(G)=q(\sigma-\sigma')\le0$.

If the edges incident to $x$ in $G$ are not in the desired form before this
operation, then we can choose $y$ and $y'$ so that $xy$ and $xy'$ both have
multiplicity at least $2$ or both have multiplicity $1$.  In either case,
the sum of the squares of the multiplicities increases in moving from $G$
to $G^o$.

As long as the graph remains $X$-surplus and some vertex of $X$ has more than
one incident edge with multiplicity at least $2$ or has more than $r$
neighbors, we can repeat this operation.  Since the sum of the squares of the
multiplicities increases with each step, the process does not cycle.  We
eventually reach the desired bigraph $G'$ in which the incident edges at each
vertex of $X$ have the desired form, or we reach a graph that is not
$X$-surplus.  To that graph we apply Case 1 to obtain $G'\in\cG_{n,k,r}$ with
$\Phi(G')\le\Phi(G)$ and edges in the desired form.
\end{proof}

The argument for Case 1 in Lemma~\ref{restrict} fails when $r=1$ because it
asks to have a positive number of copies of $r-1$ edges.  Indeed,
$X,Y$-bigraphs in $\cG_{n,k,1}$ with fewest $X$-matchings do not have the
desired structure when $k>1$.

In order to prove Theorem~\ref{main}, we first consider the case where
$G$ has an $X$-matching.

\bigskip
\begin{theorem}\label{mainproof}
\noindent
{\it Fix $k,r,n\in\NN$ with $k\ge r$.
If $G\in\cG_{n,k,r}$, then
$$
\Phi(G)\ge\begin{cases}
   r!(k-r+1) &\text{if $n\ge r$,}\\
   [r+n(k-r)]\PE i1{n-1}(r-i) &\text{if $n\le r$.}
\end{cases}
$$
Furthermore, the bounds are sharp in all cases.
}
\end{theorem}
\begin{proof}
When $r=1$, we have $n\ge r$, and the claimed lower bound is $k$.
We are guaranteed a vertex $x\in X$ having one neighbor with multiplicity $k$.
Every $X$-matching uses a copy of that edge, for which there are $k$ choices.
Hence we may assume $r\ge2$.

We use induction on $n+k$.
When $k=r$, the formulas reduce to
$\PE i0{\min\{k,n\}-1}(k-i)$, our restricted graph $G$ is simple, and the lower
bound holds from the result of M.~Hall.  When $n=1\le r$, the empty product
$\PE i1{n-1}(r-i)$ is $1$, and the number of $X$-matchings guaranteed is $k$.
Hence we may assume $n>1$ and $k>r$.

\smallskip
{\bf Case 1:} {\it $\C{N(S)}=\C S$ for some nonempty proper subset $S$ of $X$.}
Every $X$-matching in $G$ matches $S$ into $N(S)$.  Each matching of $S$ into
$N(S)$ can be completed to an $X$-matching by some matching of $X-S$ into
$Y-N(S)$, because we are given $\Phi(G)\ge1$.  Hence it suffices
to show that the subgraph $G'$ induced by $S\cup N(S)$ has at least the desired
number of $S$-matchings.  With $s=\C S$, this graph lies in $\cG_{s,k,r}$.

Since every vertex of $X$ has at least $r$ neighbors, $n> s\ge r$.  Now the
induction hypothesis for $s$ yields $\Phi(G')\ge r!(k-r+1)$ 
and hence $\Phi(G)\ge r!(k-r+1)$, as desired.
Note that Case 1 {\it cannot occur} if $n\le r$.

\smallskip
{\bf Case 2:} {\it $G$ is $X$-surplus.}
As noted earlier, we may assume $k>r$.  In order to prove the lower bound on
$\Phi(G)$, we may assume that $G$ has the form guaranteed by
Lemma~\ref{restrict}.  Let $G'$ be the $X,Y$-bigraph obtained from $G$
by deleting from $G$ one copy of each edge having multiplicity at least $2$.
Each vertex of $X$ loses one incident edge, but neighborhoods are unchanged,
so $G\in \cG_{n,k-1,r}$.  By the induction hypothesis, $\Phi(G')\ge r!(k-r)$ if
$n\ge r$ and $\Phi(G')\ge [r+n(k-1-r)]\PE i1{n-1}(r-i)$ if $n\le r$.

This lower bound does not count the $X$-matchings in $G$ that use at least one
of the deleted edges, which are copies of edges that remain.  We need to find
$r!$ such matchings if $n\ge r$, and we need to find $n\PE i1{n-1}(r-i)$ such
matchings if $n\le r$.

Since $G$ is $X$-surplus, deleting the endpoints of any edge from $G'$ leaves
$\C{N(S)}\ge\C S$ for any subset $S$ of $X$ that remains, and every remaining
vertex of $X$ keeps at least $r-1$ neighbors.  The result of M.~Hall thus
implies that every edge of $G'$ appears in at least $q$ $X$-matchings, where
$q=\PE i1{\min\{r-1,n-1\}}(r-i)$.  Since each edge deleted from $G$ is a copy
of an edge in $G'$, we obtain $q$ $X$-matchings in $G$ that use only that one
missing edge plus edges in $G'$.  Thus we have at least $nq$ distinct
$X$-matchings in $G$ that do not lie in $G'$.  In fact, $nq$ is exactly the
desired value when $n\le r$, and it exceeds the desired value when $n>r$.
\end{proof}

To complete the proof of Theorem~\ref{main}, we consider the case where $G$
does not have an $X$-matching.  In this case we obtain a larger lower bound.
Let $\alpha'(G)$ be the maximum size (number of edges) of a matching in a 
(multi)graph $G$.  For a subset $S$ of $X$ in $G$, the {\it defect} is
$\max\{0,\C S-\C{N(S)}\}$.  The Defect Formula of Ore~\cite{Ore}, proved by
the same technique as the corollary below, states for an $X,Y$-bigraph $G$ that
$\alpha'(G)=\C X-p$, where $p$ is the maximum defect among subsets of $X$.

\begin{corollary}\label{maincor}
Fix $k,r,n\in\NN$ with $k\ge r$.  Let $G$ be an $X,Y$-bigraph with
$\C X=n$, $\delta_X(G)\ge k$, and $\C{N_G(x)}\ge r$ for all $x\in X$.
Let $p=n-\alpha'(G)$.  If $p>0$, then $r<n$ and
$$
\Phi(G)\ge (k-r+1)(r+p)!/p!.
$$
Furthermore, the bound is sharp in all cases.
\end{corollary}
\begin{proof}
Form $G'$ from $G$ by adding $p$ ``universal'' vertices to $Y$, adjacent via
edges of multiplicity $1$ to all vertices of $X$.  Since this adds $p$ vertices
to the neighborhood of each subset of $X$, we have $\alpha'(G')=n$.

For each $X$-matching in $G'$, deleting the edges covering the $p$ added
vertices yields a maximum matching in $G$.  Furthermore, each maximum matching
in $G$ corresponds to $p!$ $X$-matchings in $G'$ in this way, since the 
uncovered vertices of $X$ in a maximum matching in $G$ can be matched to the
$p$ added vertices in any order.  Thus $\Phi(G)=\Phi(G')/p!$.

Since the added vertices increase the neighborhood of each subset of $X$ and
the degree of each vertex of $X$ by $p$, we have
$G'\in\cG_{n,k+p,r+p}$.  Now Theorem~\ref{mainproof} completes the proof.
\end{proof}

A closer look at the proof of Theorem~\ref{mainproof} leads to a description
of the graphs achieving equality in the bound.

\begin{theorem}\label{maineq}
When $r,n>1$ and $G$ is an $X,Y$-bigraph in $\cG_{n,k,r}$ with no isolated
vertices, equality in the bound in Theorem~\ref{mainproof} occurs only in the
following situations.

If $r\ge n$, then $G$ is $X$-surplus, all multiedges are incident to a single
vertex of $Y$, and the underlying simple graph is $K_{n,r}$.

If $r<n$, then $\C Y=\C X$ and $G$ is not $X$-surplus, and for a smallest
nonempty set $S\esub X$ such that $\C{N(S)}=\C S$, the size of $s$ is $r$,
all multiedges in $G[S\cup N(S)]$ are incident to a single vertex of $Y$, the
underlying graph of $G[S\cup N(S)]$ is $K_{r,r}$, and the subgraph
$G-(S\cup N(S))$ has exactly one $(X-S)$-matching.
\end{theorem}
\begin{proof}
We use induction on $k$.
When $n=1$, the minimum value $k$ is achieved by all arrangements of $k$
edges at the single vertex of $X$, so we restrict to $n\ge2$.
Consider $G\in \cG_{n,k,r}$ achieving equality in the bound
from Theorem~\ref{mainproof}.

\smallskip
{\bf Case ``2'':} {\it $G$ is $X$-surplus.}
Since $G$ is $X$-surplus, every edge is in an $X$-matching.  Hence equality
in the bound requires every vertex to have degree $k$.

When $k=r$, we have a simple graph and the lower bound
$\PE i0{\min\{k,n\}-1}(k-i)$ of M.~Hall.  After choosing mates for $i$ vertices
of $X$, there are at least $k-i$ neighbors available for the mate of the next
vertex of $X$.  This proves the lower bound, but also equality will hold only
if there are no more than $k-i$ neighbors available.  This requires that, no
matter how the vertices are ordered, the neighbors chosen for the earlier
vertices must be neighbors of the later vertices.  Hence all vertices of $X$
have the same neighborhood, so $G=K_{n,r}$.  Since $G$ has an $X$-matching,
this outcome requires $r\ge n$.  The desired conclusion holds.

Now suppose $k>r$.  By Lemma~\ref{restrict}, our extremal graph $G$ has the
same number of $X$-matchings as a graph $G'$ in which every vertex of $X$ has
$r$ neighbors and degree $k$, including one neighbor along an edge with
multiplicity $k-r+1$.  Also, $G'$ arises from $G$ by the steps of shifting in
the proof of Lemma~\ref{restrict}, always remaining $X$-surplus, since a
shifting step where the property of being $X$-surplus is lost strictly
decreases the number of $X$-matchings.

Let $G^*$ be the graph obtained from $G'$ by deleting one copy of each edge
having multiplicity at least $2$.  We may have (1) $X$-matchings that lie in
$G^*$ as guaranteed by the induction hypothesis, (2) $X$-matchings containing
exactly one of the missing edges, and (3) $X$-matchings containing more than
one of the missing edges.  As observed in the proof of Theorem~\ref{mainproof},
when $n>r$ types (1) and (2) already provide more $X$-matchings than the lower
bound, so we may assume $r\ge n$.  In this case, types (1) and (2) provide as
much as the lower bound, so equality forbids $X$-matchings of type (3).

Suppose that the multiedges at $x_1$ and $x_2$ do not have the same endpoint
in $Y$, so that $x_1y_1$ and $x_2y_2$ are disjoint edges in $G'$ with
multiplicity at least $2$.  Let $X'=X-\{x_1,x_2\}$.  If $\C{N(S)}\ge\C S+2$ for
all nonempty $S\esub X'$, then $G'-\{x_1,x_2,y_1,y_2\}$ satisfies Hall's
Condition and has an $X'$-matching, yielding an $X$-matching of type (3) in
$G'$.  Hence $\C{N(S)}\le \C S+1$ for some nonempty $S\esub X'$.  Since $G$ is
$X$-surplus, equality holds.  Now
$$n-1\ge\C S+1=\C{N(S)}\ge r.$$
We conclude $r<n$, which contradicts the earlier conclusion $r\ge n$ for
this case.  Hence the assumption about $x_1$ and $x_2$ cannot hold, and in
fact all the multiedges in $G'$ have the same endpoint $y$ in $Y$.

Thus when $G$ is extremal, the shifting process of Lemma~\ref{restrict} turns
$G$ into a graph $G'$ satisfying the properties in Lemma~\ref{restrict} plus
the property that the multiedges have the same endpoint in $Y$, all without
changing the number of $X$-matchings.  Consider the last shifting step, which
brings vertex $x$ into compliance.  Already the multiedges at the vertices of
$X-\{x\}$ have common endpoint $y$.  Since the last shifting step results in
$k-r+1$ edges from $x$ to $y$, we must already have $xy$ as an edge, and the
extra copies are coming from $xy'$.  Since $G$ is extremal,
$\Phi(G-\{x,y\})=\Phi(G-\{x,y'\})$.

To complete the proof, we take a closer look at $G'$.  Deleting $y$ yields a
simple graph in which every vertex of $X$ has degree $r-1$.  Hence
$G'$ has at least $\PE i0{n-1}(r-1-i)$ $X$-matchings that avoid $y$
(this lower bound equals $0$ when $n=r$).  For $X$-matchings covering $y$, we
pick the edge covering $y$ in $n(k-r+1)$ ways and complete the matching in a
simple $X',Y'$-bigraph where $\C{X'}=n-1$ and every vertex of $X'$ has degree
$r-1$.  Hence there are at least $n(k-r+1)\PE i0{n-2}(r-1-i)$ $X$-matchings
covering $y$.  Together, we have the desired lower bound
$[r+n(k-r)]\PE i1{n-1}(r-i)$ from Theorem~\ref{mainproof}.

If equality holds, then equality must hold for each contribution, using $y$
or not using $y$.  Each of those contributions came by counting matchings
in a simple bigraph in which the vertices in the first part all have the same
degree.  By the argument for the case $k=r$ at the beginning of this Case 2,
equality in the bound requires the underlying simple graph of $G'$ to be
$K_{n,r}$.  In $G'$, each edge incident to $y$ lies in $\PE i1{n-1}(r-i)$
$X$-matchings, and any edge at $x$ not incident to $y$ lies in 
$[(n-1)(k-r+1)+(r-n)]\PE i1{n-2}(r-1-i)$ $X$-matchings.  Since
$(n-1)(k-r)+(r-1)>(r-1)$, in fact $\Phi(G-\{x,y\})<\Phi(G-\{x,y'\})$ and the
last shifting step to reach $G'$ reduces $\Phi$.  Thus $G$ must in fact have
the form of $G'$.

\smallskip
{\bf Case ``1'':} {\it $G$ is not $X$-surplus.}
Let $S$ with size $s$ be a smallest nonempty subset of $X$ such that
$\C{N(S)}=\C{S}$.  We have already observed that this case requires $n>r$ and
has $r!(k-r+1)$ $S$-matchings.  Let $G'=G[S\cup N(S)]$.  By the choice of
$S$, $G'$ is $S$-surplus, so the conclusions of Case 2 apply to it.
All multiedges in $G'$ are incident to a single vertex of $Y$, and $r\ge s$.
Also, $r\le\C{N(S)}$, so $r=s$ and the underlying graph of $G'$ is $K_{r,r}$.

In the subgraph $G''$ obtained by deleting $S\cup N(S)$, equality in $\Phi(G)$
requires that there is only one $(X-S)$-matching.  Since $G$ has no isolated
vertex, this means that an $(X-S)$-matching cannot leave a vertex of $Y-N(S)$
uncovered; shifting an edge would produce another $(X-S)$-matching.
Hence $\C Y=\C X=n$.  Now M. Hall's formula implies that each of $X-S$ and
$Y-N(S)$ must contain a vertex of degree $1$ in $G''$.  Other edges may
be added in various ways, and there may be many edges joining $X-S$ to $N(S)$.
\end{proof}

\section{Elementary Graphs}\label{sec:surp}

In this section we consider imposing both a neighborhood restriction on
vertices of $X$ and a degree restriction on vertices of $Y$, proving the
special case of Theorem~\ref{4k} where $\C Y=\C X$.  Our proof uses properties
of the ``elementary'' graphs mentioned in the introduction.

There are many equivalent characterizations of the connected bipartite graphs
in which all edges are {\it allowed}, meaning that they appear in a perfect
matching.  These equivalences can be found in Lov\'asz and Plummer~\cite{LP},
though most of the results appeared already in Hetyei~\cite{He} and/or
Lov\'asz~\cite{Lov}.

The $2$-connected graphs are characterized by the existence of ear
decompositions.  An {\it ear} in a graph is a path whose internal vertices have
degree $2$.  An {\it ear decomposition} iteratively deletes an ear (except
that the endpoints of the ear stay) until what remains is only a cycle.  An
ear may be a single edge, so it suffices to begin with a cycle and add ears to
obtain a spanning subgraph (and multiedges are irrelevant).  An {\it odd ear}
is an ear of odd length, and an {\it odd ear decomposition} uses only odd ears.
In a bipartite graph, the endpoints of an odd ear lie in opposite parts.  Also,
since all cycles are even, we may view an odd ear decomposition of a bipartite
graph as starting from an edge and adding paths of odd length.  This allows
$K_{1,1}$ to have an odd ear decomposition.

\begin{theorem}\label{elem}
For an $X,Y$-bigraph $G$ with a perfect matching, the following are equivalent.\\
(a) The subgraph consisting of allowed edges is connected (i.e., $G$ is
``elementary'').\\
(b) $\C{N(S)}>\C S$ whenever $\nul\ne S\subsetneq X$ (i.e., $G$ is
near-surplus).\\
(c) $G-x-y$ has a perfect matching, whenever $x\in X$ and $y\in Y$ (i.e., $G$
is ``bicritical''.\\
(d) $G$ is connected and every edge is allowed.\\
(e) $G$ has an odd ear decomposition.
\end{theorem}

The implications (a)$\Imp$(b)$\Imp$(c)$\Imp$(d)$\Imp$(a) follow immediately.
Lov\'asz and Plummer~\cite{LP} presented an accessible argument for building an
odd ear decomposition of $G$ from any given edge, using property (d).
It is easy to prove that (e) implies the other properties.

\begin{theorem}\label{surplust}
Every elementary $X,Y$-bigraph $G$ has at least $\C{E(G)}-\C{V(G)}+2$ perfect
matchings, and the bound is sharp.
\end{theorem}
\begin{proof}
Let $m=\C{E(G)}$ and $n=\C{V(G)}$; we use induction on $m-n$.  When $m-n=0$,
$G$ is an even cycle, and there are two perfect matchings.  When $m-n>0$,
we consider the odd ear decomposition provided by Theorem~\ref{elem}(e).
Let $G'$ be the graph obtained by removing one ear $P$ in the decomposition,
having endpoints $x$ and $y$.  This removes one more edge than vertex.  By the
induction hypothesis, $G'$ has at least $m-n+1$ perfect matchings.  Each
extends by a matching of the internal vertices of $P$ to a perfect matching
of $G$.

To produce one more perfect matching in $G$, start with a perfect matching $M$
in the ear, covering $x$ and $y$ and any vertices between them.  Since $G$ is
elementary, $G'-\{x,y\}$ has a perfect matching $M'$.  Now $M\cup M'$ is a
perfect matching of $G$ not counted among the $m-n+1$ matchings above.

For sharpness, let $G$ be a union of paths with odd length, all having
endpoints $x$ and $y$.  In any perfect matching $M$, the path containing the
edge of $M$ covering $x$ provides also the edge covering $y$, since the path
has odd length.  In all other paths, $M$ covers the internal vertices without
covering $x$ or $y$.  Hence the number of perfect matchings is the number of
paths with endpoints $x$ and $y$.  After forming a cycle using two of the
paths, contributes the same number of vertices and edges, each additional
path adds one more edge than vertex.  Hence the number of paths (and matchings)
is $2+m-n$, as desired.  See Figure~\ref{earfig}.
\end{proof}

\vspace{-1pc}
\begin{centering}
\begin{figure}[h]
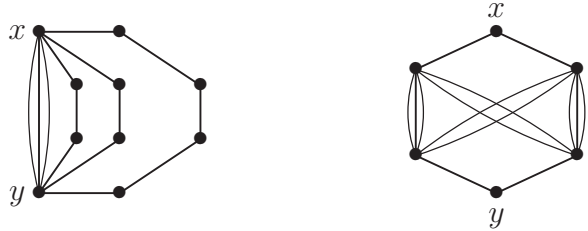

\gpic{
\expandafter\ifx\csname graph\endcsname\relax \csname newbox\endcsname\graph\fi
\expandafter\ifx\csname graphtemp\endcsname\relax \csname newdimen\endcsname\graphtemp\fi
\setbox\graph=\vtop{\vskip 0pt\hbox{%
    \graphtemp=.5ex\advance\graphtemp by 0.958in
    \rlap{\kern 0.113in\lower\graphtemp\hbox to 0pt{\hss $\bu$\hss}}%
    \graphtemp=.5ex\advance\graphtemp by 0.113in
    \rlap{\kern 0.113in\lower\graphtemp\hbox to 0pt{\hss $\bu$\hss}}%
    \graphtemp=.5ex\advance\graphtemp by 0.676in
    \rlap{\kern 0.310in\lower\graphtemp\hbox to 0pt{\hss $\bu$\hss}}%
    \graphtemp=.5ex\advance\graphtemp by 0.394in
    \rlap{\kern 0.310in\lower\graphtemp\hbox to 0pt{\hss $\bu$\hss}}%
    \graphtemp=.5ex\advance\graphtemp by 0.958in
    \rlap{\kern 0.535in\lower\graphtemp\hbox to 0pt{\hss $\bu$\hss}}%
    \graphtemp=.5ex\advance\graphtemp by 0.113in
    \rlap{\kern 0.535in\lower\graphtemp\hbox to 0pt{\hss $\bu$\hss}}%
    \graphtemp=.5ex\advance\graphtemp by 0.676in
    \rlap{\kern 0.535in\lower\graphtemp\hbox to 0pt{\hss $\bu$\hss}}%
    \graphtemp=.5ex\advance\graphtemp by 0.394in
    \rlap{\kern 0.535in\lower\graphtemp\hbox to 0pt{\hss $\bu$\hss}}%
    \graphtemp=.5ex\advance\graphtemp by 0.676in
    \rlap{\kern 0.958in\lower\graphtemp\hbox to 0pt{\hss $\bu$\hss}}%
    \graphtemp=.5ex\advance\graphtemp by 0.394in
    \rlap{\kern 0.958in\lower\graphtemp\hbox to 0pt{\hss $\bu$\hss}}%
    \graphtemp=.5ex\advance\graphtemp by 0.958in
    \rlap{\kern 0.000in\lower\graphtemp\hbox to 0pt{\hss $y$\hss}}%
    \graphtemp=.5ex\advance\graphtemp by 0.113in
    \rlap{\kern 0.000in\lower\graphtemp\hbox to 0pt{\hss $x$\hss}}%
    \special{pn 11}%
    \special{pa 113 958}%
    \special{pa 113 113}%
    \special{fp}%
    \special{pn 8}%
    \special{ar -1523 535 1690 1690 -0.252680 0.252680}%
    \special{ar 1749 535 1690 1690 2.888912 3.394273}%
    \special{pn 11}%
    \special{pa 113 958}%
    \special{pa 310 676}%
    \special{fp}%
    \special{pa 310 676}%
    \special{pa 310 394}%
    \special{fp}%
    \special{pa 310 394}%
    \special{pa 113 113}%
    \special{fp}%
    \special{pa 113 958}%
    \special{pa 535 676}%
    \special{fp}%
    \special{pa 535 676}%
    \special{pa 535 394}%
    \special{fp}%
    \special{pa 535 394}%
    \special{pa 113 113}%
    \special{fp}%
    \special{pa 113 958}%
    \special{pa 535 958}%
    \special{fp}%
    \special{pa 535 958}%
    \special{pa 958 676}%
    \special{fp}%
    \special{pa 958 676}%
    \special{pa 958 394}%
    \special{fp}%
    \special{pa 958 394}%
    \special{pa 535 113}%
    \special{fp}%
    \special{pa 535 113}%
    \special{pa 113 113}%
    \special{fp}%
    \graphtemp=.5ex\advance\graphtemp by 0.761in
    \rlap{\kern 2.085in\lower\graphtemp\hbox to 0pt{\hss $\bu$\hss}}%
    \graphtemp=.5ex\advance\graphtemp by 0.310in
    \rlap{\kern 2.085in\lower\graphtemp\hbox to 0pt{\hss $\bu$\hss}}%
    \graphtemp=.5ex\advance\graphtemp by 0.958in
    \rlap{\kern 2.507in\lower\graphtemp\hbox to 0pt{\hss $\bu$\hss}}%
    \graphtemp=.5ex\advance\graphtemp by 0.113in
    \rlap{\kern 2.507in\lower\graphtemp\hbox to 0pt{\hss $\bu$\hss}}%
    \graphtemp=.5ex\advance\graphtemp by 0.761in
    \rlap{\kern 2.930in\lower\graphtemp\hbox to 0pt{\hss $\bu$\hss}}%
    \graphtemp=.5ex\advance\graphtemp by 0.310in
    \rlap{\kern 2.930in\lower\graphtemp\hbox to 0pt{\hss $\bu$\hss}}%
    \special{pa 2085 761}%
    \special{pa 2507 958}%
    \special{fp}%
    \special{pa 2507 958}%
    \special{pa 2930 761}%
    \special{fp}%
    \special{pa 2930 761}%
    \special{pa 2930 310}%
    \special{fp}%
    \special{pa 2930 310}%
    \special{pa 2507 113}%
    \special{fp}%
    \special{pa 2507 113}%
    \special{pa 2085 310}%
    \special{fp}%
    \special{pa 2085 310}%
    \special{pa 2085 761}%
    \special{fp}%
    \graphtemp=.5ex\advance\graphtemp by 1.070in
    \rlap{\kern 2.507in\lower\graphtemp\hbox to 0pt{\hss $y$\hss}}%
    \graphtemp=.5ex\advance\graphtemp by 0.000in
    \rlap{\kern 2.507in\lower\graphtemp\hbox to 0pt{\hss $x$\hss}}%
    \special{pn 8}%
    \special{ar 1477 535 648 648 -0.355251 0.355251}%
    \special{ar 2692 535 648 648 2.786341 3.496844}%
    \special{ar 2322 535 648 648 -0.355251 0.355251}%
    \special{ar 3537 535 648 648 2.786341 3.496844}%
    \special{ar 1201 -1913 2817 2817 0.910009 1.251669}%
    \special{ar 3813 2985 2817 2817 -2.231583 -1.889924}%
    \special{ar 3813 -1913 2817 2817 1.889924 2.231583}%
    \special{ar 1201 2985 2817 2817 -1.251669 -0.910009}%
    \hbox{\vrule depth1.070in width0pt height 0pt}%
    \kern 3.000in
  }%
}%
}

\vspace{-1pc}
\caption{Constructions for Theorem~\ref{surplust}.\label{earfig}}
\end{figure}
\end{centering}

Another sharpness example for Theorem~\ref{surplust} is obtained from a copy of
$K_{3,3}$ by deleting one edge $xy$ and then allowing arbitrary multiplicity
for each of the four edges not incident to $x$ or $y$.  Each edge $e$ not
incident to $x$ or $y$ determines a perfect matching, matching $x$ and $y$ to
the two remaining vertices that are not incident to $x$ or $y$, and every
perfect matching has exactly one such edge.  Hence the number of perfect
matchings is the number of edges not incident to $x$ or $y$.  Since there are
four edges incident to $x$ or $y$, the number of perfect matchings is $m-4$,
which equals $m-n+2$.  See Figure~\ref{earfig}.

We apply Theorem~\ref{surplust} to prove a lower bound on the number of
$X$-matchings where we have requirements on degrees in both $X$ and $Y$ and
on neighborhoods of vertices in $X$.  First we present sharpness constructions.

\begin{construction}\label{4kcon}
For $n\ge4$, let $J_{n,k}$ be the $X,Y$-bigraph with $\C X=\C Y=n$ consisting
of the disjoint union of $F_k$ and a cycle $C$ with $2n-4$ vertices, plus $k-2$
edges from $y_1$ in $F_k$ to each vertex of $X$ on $C$.  Note that
$\delta_X(J_{n,k})=k$, and each vertex of $J_{n,k}$ has at least two neighbors.
An $X$-matching must consist of an perfect matching in $F_k$ and a perfect
matching in $C$.  There are $2k-2$ of the former and two of the latter,
so $\Phi(J_{n,k}=4k-4$.  See Figure~\ref{4kfig}.

Form $H''_{n,k}$ from $H_{n,k}$ (defined in Construction~\ref{sharp2}) by
shifting one copy of $x_1y_1$ to $x_1y_i$, for each $i$ with $2\le i\le n$.
This requires $k\ge n$.  The shifted edges ensure that each vertex has at least
two neighbors.  For $i\ge2$, the only neighbors of $x_i$ are $y_i$ and $y_1$,
with multiplicity $1$ and $k-1$, respectively, and the only neighbors of $y_i$
are $x_i$ and $x_1$.  An $X$-matching must be a perfect matching.  Once the
edge incident to $y_1$ is chosen, the rest of the matching is determined.
Hence $\Phi(H''_{n,k})= (n-1)(k-1)+(k-n+1)= n(k-2)+2$.  See Figure~\ref{4kfig}.
\end{construction}

\begin{centering}
\begin{figure}[h]\label{4kfig}
\gpic{
\expandafter\ifx\csname graph\endcsname\relax \csname newbox\endcsname\graph\fi
\expandafter\ifx\csname graphtemp\endcsname\relax \csname newdimen\endcsname\graphtemp\fi
\setbox\graph=\vtop{\vskip 0pt\hbox{%
    \graphtemp=.5ex\advance\graphtemp by 0.903in
    \rlap{\kern 0.159in\lower\graphtemp\hbox to 0pt{\hss $\bu$\hss}}%
    \graphtemp=.5ex\advance\graphtemp by 0.637in
    \rlap{\kern 0.159in\lower\graphtemp\hbox to 0pt{\hss $\bu$\hss}}%
    \graphtemp=.5ex\advance\graphtemp by 0.372in
    \rlap{\kern 0.159in\lower\graphtemp\hbox to 0pt{\hss $\bu$\hss}}%
    \graphtemp=.5ex\advance\graphtemp by 0.106in
    \rlap{\kern 0.159in\lower\graphtemp\hbox to 0pt{\hss $\bu$\hss}}%
    \graphtemp=.5ex\advance\graphtemp by 0.903in
    \rlap{\kern 0.956in\lower\graphtemp\hbox to 0pt{\hss $\bu$\hss}}%
    \graphtemp=.5ex\advance\graphtemp by 0.637in
    \rlap{\kern 0.956in\lower\graphtemp\hbox to 0pt{\hss $\bu$\hss}}%
    \graphtemp=.5ex\advance\graphtemp by 0.372in
    \rlap{\kern 0.956in\lower\graphtemp\hbox to 0pt{\hss $\bu$\hss}}%
    \graphtemp=.5ex\advance\graphtemp by 0.106in
    \rlap{\kern 0.956in\lower\graphtemp\hbox to 0pt{\hss $\bu$\hss}}%
    \special{pn 11}%
    \special{pa 159 903}%
    \special{pa 956 903}%
    \special{fp}%
    \special{pa 956 903}%
    \special{pa 159 637}%
    \special{fp}%
    \special{pa 159 637}%
    \special{pa 956 637}%
    \special{fp}%
    \special{pa 956 637}%
    \special{pa 159 903}%
    \special{fp}%
    \special{pa 159 372}%
    \special{pa 956 372}%
    \special{fp}%
    \special{pa 956 372}%
    \special{pa 159 106}%
    \special{fp}%
    \special{pa 159 106}%
    \special{pa 956 106}%
    \special{fp}%
    \special{pa 956 106}%
    \special{pa 159 372}%
    \special{fp}%
    \special{pn 8}%
    \special{ar 558 -1435 1593 1593 1.318116 1.823477}%
    \special{ar 558 1649 1593 1593 -1.823477 -1.318116}%
    \special{ar -57 -1607 1991 1991 1.036634 1.461458}%
    \special{ar 1173 2085 1991 1991 -2.104959 -1.680135}%
    \special{pn 11}%
    \special{pa 159 637}%
    \special{pa 956 106}%
    \special{fp}%
    \special{pn 8}%
    \special{ar -890 -1800 2655 2655 0.801525 1.164062}%
    \special{ar 2006 2544 2655 2655 -2.340068 -1.977530}%
    \special{pn 11}%
    \special{pa 159 903}%
    \special{pa 956 106}%
    \special{fp}%
    \special{pn 8}%
    \special{ar -2229 -2282 3982 3982 0.643501 0.927295}%
    \special{ar 3345 3292 3982 3982 -2.498092 -2.214297}%
    \graphtemp=.5ex\advance\graphtemp by 0.903in
    \rlap{\kern 0.000in\lower\graphtemp\hbox to 0pt{\hss $x_1$\hss}}%
    \graphtemp=.5ex\advance\graphtemp by 0.106in
    \rlap{\kern 0.000in\lower\graphtemp\hbox to 0pt{\hss $x_n$\hss}}%
    \graphtemp=.5ex\advance\graphtemp by 0.106in
    \rlap{\kern 1.142in\lower\graphtemp\hbox to 0pt{\hss $y_1$\hss}}%
    \graphtemp=.5ex\advance\graphtemp by 0.903in
    \rlap{\kern 1.142in\lower\graphtemp\hbox to 0pt{\hss $y_n$\hss}}%
    \graphtemp=.5ex\advance\graphtemp by 0.903in
    \rlap{\kern 2.018in\lower\graphtemp\hbox to 0pt{\hss $\bu$\hss}}%
    \graphtemp=.5ex\advance\graphtemp by 0.637in
    \rlap{\kern 2.018in\lower\graphtemp\hbox to 0pt{\hss $\bu$\hss}}%
    \graphtemp=.5ex\advance\graphtemp by 0.372in
    \rlap{\kern 2.018in\lower\graphtemp\hbox to 0pt{\hss $\bu$\hss}}%
    \graphtemp=.5ex\advance\graphtemp by 0.106in
    \rlap{\kern 2.018in\lower\graphtemp\hbox to 0pt{\hss $\bu$\hss}}%
    \graphtemp=.5ex\advance\graphtemp by 0.903in
    \rlap{\kern 2.814in\lower\graphtemp\hbox to 0pt{\hss $\bu$\hss}}%
    \graphtemp=.5ex\advance\graphtemp by 0.637in
    \rlap{\kern 2.814in\lower\graphtemp\hbox to 0pt{\hss $\bu$\hss}}%
    \graphtemp=.5ex\advance\graphtemp by 0.372in
    \rlap{\kern 2.814in\lower\graphtemp\hbox to 0pt{\hss $\bu$\hss}}%
    \graphtemp=.5ex\advance\graphtemp by 0.106in
    \rlap{\kern 2.814in\lower\graphtemp\hbox to 0pt{\hss $\bu$\hss}}%
    \special{pn 11}%
    \special{pa 2018 903}%
    \special{pa 2814 903}%
    \special{fp}%
    \special{pa 2018 637}%
    \special{pa 2814 637}%
    \special{fp}%
    \special{pa 2018 372}%
    \special{pa 2814 372}%
    \special{fp}%
    \special{pa 2018 106}%
    \special{pa 2814 106}%
    \special{fp}%
    \special{pa 2018 106}%
    \special{pa 2814 903}%
    \special{fp}%
    \special{pa 2018 106}%
    \special{pa 2814 637}%
    \special{fp}%
    \special{pa 2018 106}%
    \special{pa 2814 372}%
    \special{fp}%
    \special{pa 2018 372}%
    \special{pa 2814 106}%
    \special{fp}%
    \special{pn 8}%
    \special{ar 1800 -1607 1991 1991 1.036634 1.461458}%
    \special{ar 3031 2085 1991 1991 -2.104959 -1.680135}%
    \special{pn 11}%
    \special{pa 2018 637}%
    \special{pa 2814 106}%
    \special{fp}%
    \special{pn 8}%
    \special{ar 967 -1800 2655 2655 0.801525 1.164062}%
    \special{ar 3864 2544 2655 2655 -2.340068 -1.977530}%
    \special{pn 11}%
    \special{pa 2018 903}%
    \special{pa 2814 106}%
    \special{fp}%
    \special{pn 8}%
    \special{ar -371 -2282 3982 3982 0.643501 0.927295}%
    \special{ar 5204 3292 3982 3982 -2.498092 -2.214297}%
    \graphtemp=.5ex\advance\graphtemp by 0.903in
    \rlap{\kern 1.858in\lower\graphtemp\hbox to 0pt{\hss $x_1$\hss}}%
    \graphtemp=.5ex\advance\graphtemp by 0.106in
    \rlap{\kern 1.858in\lower\graphtemp\hbox to 0pt{\hss $x_n$\hss}}%
    \graphtemp=.5ex\advance\graphtemp by 0.106in
    \rlap{\kern 3.000in\lower\graphtemp\hbox to 0pt{\hss $y_1$\hss}}%
    \graphtemp=.5ex\advance\graphtemp by 0.903in
    \rlap{\kern 3.000in\lower\graphtemp\hbox to 0pt{\hss $y_n$\hss}}%
    \graphtemp=.5ex\advance\graphtemp by 1.168in
    \rlap{\kern 0.558in\lower\graphtemp\hbox to 0pt{\hss $J_{4,4}$\hss}}%
    \graphtemp=.5ex\advance\graphtemp by 1.168in
    \rlap{\kern 2.416in\lower\graphtemp\hbox to 0pt{\hss $H''_{4,4}$\hss}}%
    \hbox{\vrule depth1.274in width0pt height 0pt}%
    \kern 3.000in
  }%
}%
}

\vspace{-1pc}
\caption{Construction~\ref{4kcon}}
\end{figure}
\end{centering}

In applying Theorem~\ref{surplust}, we restrict to $\C Y=\C X$, but the results
of the next section will apply to the case $\C Y>\C X$.

\bigskip
\noindent
{\bf Theorem~\ref{4k}.}
{\it Let $k,n\geq 2$.
Let $G$ be an $X,Y$-bigraph with $\C X=\C Y=n$ having an $X$-matching.  If
$\delta_X(G)\geq k$, $\delta_Y(G)\geq 2$ and $|N_G(x)|\geq 2$ for each
$x\in X$, then
\begin{equation*}\label{e1}
\Phi(G)\geq \min\{n(k-2)+2, 4k-4\}.
\end{equation*}
}
  
\begin{proof}
{\bf Case 1:} {\it $\C{N(S)}=\C S$ for some nonempty proper subset $S$ of $X$.}
Let $G_1$ be the subgraph of $G$ induced by $S\cup N(S)$, and let
$G_2=G-V(G_1)$.  By Theorem~\ref{main}, $\Phi(G_1)\ge 2k-2$.
Since $\delta_Y(G)\ge2$ and vertices of $S$ have no neighbors in $Y-N(S)$,
we have $\delta_Y(G_2)\ge2$.  Also $\Phi(G_2)\ge1$, because $\Phi(G)\ge1$
and $S$ can only match into $N(S)$.  Hence we can apply Theorem~\ref{main}
to $G_2$ as a $(Y-N(S),X-S)$-bigraph to find at least two perfect matchings
in $G_2$.  Combining these with perfect matchings from $G_1$ yields
$\Phi(G)\ge 4k-4$.

{\bf Case 2:} {\it $\C{N(S)}>\C S$ for every nonempty proper subset $S$ of $X$.}
In this case, $G$ is elementary, and Theorem~\ref{surplust} provides
$\C{E(G)}-\C{V(G)}+2$ $X$-matchings.  Since $\C{E(G)}\ge nk$ and $\C{V(G)}=2n$,
we obtain the desired lower bound $n(k-2)+2$.
\end{proof}

\section{Neighborhood Restrictions}\label{sec:quad}

Recall that an $X,Y$-bigraph is {\it $X$-surplus} if $\C{N(S)}>\C S$ for every
nonempty proper subset $S$ of $X$.  A multigraph is {\it leafless} if every
vertex has at least two neighbors.

\begin{lemma}[{\rm\cite{LL}, Proposition 6}]\label{t+1}
If $G$ is an $X,Y$-bigraph without isolated vertices satisfying Hall's Condition
for $X$, and $\C Y-\C X=t$, then $\Phi(G)\ge t+1$, and this is sharp.
\end{lemma}
\begin{proof}
Note that $\C Y\ge \C X$.
Let $M$ be an $X$-matching in $G$, and let $T$ be the subset of $Y$ covered
by $M$.  For each vertex $y\in Y-T$, there is an incident edge $xy$.  Another
$X$-matching is obtained by using $xy$ to replace the edge covering $x$ in $M$.
This generates $t$ additional $X$-matchings.

Sharpness is achieved by the $X,Y$-bigraph with $X=\{\VEC x1n\}$ and
$Y=\{\VEC y1{n+t}\}$ whose edges are $\{x_iy_i\st 1\le i\le n\}$ plus
$\{x_1y_i\st n<i\le n+t\}$.
\end{proof}

\begin{corollary}\label{t+1cor}
In a leafless $X$-surplus $X,Y$-bigraph $G$, every edge appears in at least
$t+1$ $X$-matchings, where $t=\C Y-\C X$.
\end{corollary}
\begin{proof}
Since $\C{N(S)}>\C S$ for every proper subset $S$ of $X$, deleting the
endpoints of any edge $xy$ in $G$ yields an $X',Y'$-bigraph $G'$ satisfying
Hall's Condition.  Hence $\Phi(G')\ge1$, and $\C{Y'}-\C{X'}=\C Y-\C X$.
Also $\delta_{Y'}(G')\ge1$.  Hence Lemma~\ref{t+1} yields $t+1$ $X'$-matchings
in $G'$, and adding $xy$ turns each into an $X$-matching in $G$.
\end{proof}

\begin{theorem}\label{case1}
Let $G$ be an $X,Y$-bigraph satisfying Hall's Condition, such that
$\delta_X(G)\ge k$ and every vertex of $G$ has at least $r$ neighbors.  Let
$t=|Y|-|X|$. If some nonempty proper subset $S$ of $X$ satisfies $|S|=|N(S)|$,
then $\Phi(G)\ge r!(k-r+1)(r+t)!/t!$, and this is sharp.
\end{theorem}
\begin{proof}
Fix $S$ to be a largest subset of $X$ whose neighborhood has the same size
as the set.  Let $S'=X-S$ and $T=Y-N(S)$, so $N(T)\subseteq S'$.
Since every vertex of $T$ has at least $r$ neighbors, $\C{S'}\ge r$, and
hence $r<\C X$.  Let $G'=G[S\cup N(S)]$.  By Theorem~\ref{main},
$\Phi(G')\ge r!(k-r+1)$, and this is sharp.  Let $G''=G[S'\cup T]$.
Since $G$ satisfies Hall's Condition, $G''$ has an $S$-matching $M$.
It suffices to show $\Phi(G'')\ge (r+t)!/t!$.

Form $G'''$ by adding $t$ vertices to $S'$, each adjacent to all of $T$
via single edges.  Note that $G'''$ is a $Y,X$-bigraph with $\C Y=\C X\ge r+t$
in which every vertex of $Y$ has at $r+t$ neighbors.  By the result of M.~Hall,
$G'''$ has at least $(r+t)!$ perfect matchings.  Each restricts to an
$S'$-matching in $G''$, and every $S'$-matching in $G''$ arises from $t!$
perfect matchings in $G'''$ by such a matching.  Hence $\Phi(G'')\ge(r+t)!/t!$.

To construct a sharpness example, let $G'$ be a sharpness example for
Theorem~\ref{main}, and let $G''=K_{r,r+t}$, with all of $S'$ adjacent to
all of $N(S)$.
\end{proof}

The situation addressed in Theorem~\ref{case1} is what we called ``Case 1''
in earlier sections.  The other possibility, which was Case 2, is that
$G$ is $X$-surplus, as in Corollary~\ref{t+1cor} and in the remainder of this
section.  An important distinction is that the lower bound in
Theorem~\ref{case1} does not grow with $n$, but generally the values for
$X$-surplus graphs will do so.  The key feature of the first construction is
that for any fixed $r$ and $t$ the dependence on $n$ is only linear,
and similarly for the dependence on the number of edges.  We will show that the
construction is optimal for $r=2$.

\begin{construction}\label{linearn}
{\it An $X$-surplus $X,Y$-bigraph $M_{n,r,t,b}$ with few $X$-matchings.}
We fix $\C X=n\ge2r$ and $\C Y-\C X=t\ge0$.  The graph will have $b+r(n+t)$
edges, where $b\ge (r-1)(n-2r)$.  Every vertex will have at least $r$
neighbors, and
$$
\Phi(M_{n,r,t,b})
=(r-1)!\FR{(r+t-1)!}{t!}\big[b+r(t+1)+(r-1)(n-2r+1)(r+t-2)\big]
$$

We begin with a simple graph.  Let $X=R\cup S\cup T\cup \{u\}$ and
$Y=R'\cup S'\cup T'\cup \{u'\}$, with $\C R=\C S=\C{S'}=r-1$, $\C{R'}=t+r-1$,
and $\C T=\C{T'}=n-2r+1$.
Let $R\cup R'$ be an independent set.
Let $N(u)=R'\cup\{u'\}$ and $N(u')=R\cup\{u\}$.
Let $S\cup R'\cup T'$ induce $K_{r-1,n+t-r}$.
Let $S'\cup R\cup T$ induce $K_{r-1,n-r}$.
Let $T\cup T'$ induce a matching with $n-2r$ disjoint edges.
See Figure~\ref{Mfig}.
To generalize, add any edges joining $S$ and $S'$; here multiedges can
be included, so there is no upper bound on the number of edges.

\begin{figure}[h]
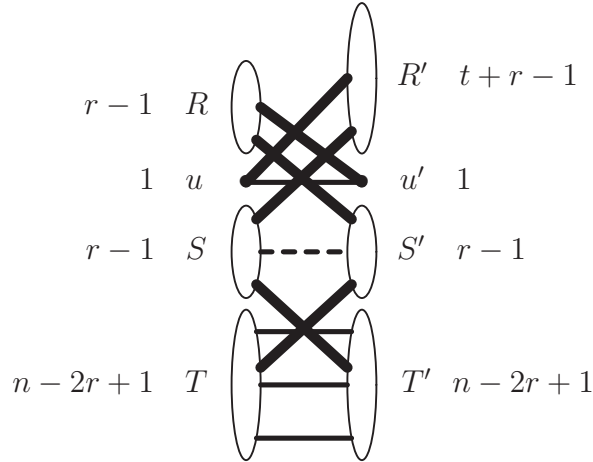

\begin{center}
\gpic{
\expandafter\ifx\csname graph\endcsname\relax \csname newbox\endcsname\graph\fi
\expandafter\ifx\csname graphtemp\endcsname\relax \csname newdimen\endcsname\graphtemp\fi
\setbox\graph=\vtop{\vskip 0pt\hbox{%
    \special{pn 11}%
    \special{ar 197 545 76 242 0 6.28319}%
    \special{ar 803 394 76 394 0 6.28319}%
    \graphtemp=.5ex\advance\graphtemp by 0.545in
    \rlap{\kern 0.000in\lower\graphtemp\hbox to 0pt{\hss $r-1~~~R\qquad\qquad$\hss}}%
    \graphtemp=.5ex\advance\graphtemp by 0.394in
    \rlap{\kern 1.000in\lower\graphtemp\hbox to 0pt{\hss $\qquad\qquad\quad~~ R'~~~t+r-1$\hss}}%
    \graphtemp=.5ex\advance\graphtemp by 0.939in
    \rlap{\kern 0.197in\lower\graphtemp\hbox to 0pt{\hss $\bu$\hss}}%
    \graphtemp=.5ex\advance\graphtemp by 0.939in
    \rlap{\kern 0.803in\lower\graphtemp\hbox to 0pt{\hss $\bu$\hss}}%
    \graphtemp=.5ex\advance\graphtemp by 0.939in
    \rlap{\kern 0.076in\lower\graphtemp\hbox to 0pt{\hss $1~~~u~~~~\qquad$\hss}}%
    \graphtemp=.5ex\advance\graphtemp by 0.939in
    \rlap{\kern 1.000in\lower\graphtemp\hbox to 0pt{\hss $\qquad~ u'~~~1$\hss}}%
    \special{ar 197 1303 76 242 0 6.28319}%
    \special{ar 803 1303 76 242 0 6.28319}%
    \graphtemp=.5ex\advance\graphtemp by 1.303in
    \rlap{\kern 0.000in\lower\graphtemp\hbox to 0pt{\hss $r-1~~~S\qquad\qquad$\hss}}%
    \graphtemp=.5ex\advance\graphtemp by 1.303in
    \rlap{\kern 1.000in\lower\graphtemp\hbox to 0pt{\hss $\qquad\qquad S'~~~r-1$\hss}}%
    \special{ar 197 2000 76 394 0 6.28319}%
    \special{ar 803 2000 76 394 0 6.28319}%
    \graphtemp=.5ex\advance\graphtemp by 2.000in
    \rlap{\kern 0.000in\lower\graphtemp\hbox to 0pt{\hss $n-2r+1~~~T~\qquad\qquad\qquad$\hss}}%
    \graphtemp=.5ex\advance\graphtemp by 2.000in
    \rlap{\kern 1.000in\lower\graphtemp\hbox to 0pt{\hss $\qquad\qquad\qquad~T'~~n-2r+1$\hss}}%
    \special{pn 56}%
    \special{pa 749 672}%
    \special{pa 251 1132}%
    \special{fp}%
    \special{pa 251 717}%
    \special{pa 749 1132}%
    \special{fp}%
    \special{pa 251 1474}%
    \special{pa 727 1909}%
    \special{fp}%
    \special{pa 749 1474}%
    \special{pa 273 1909}%
    \special{fp}%
    \special{pn 28}%
    \special{pa 273 1303}%
    \special{pa 727 1303}%
    \special{da 0.061}%
    \special{pa 251 1721}%
    \special{pa 749 1721}%
    \special{fp}%
    \special{pa 251 2279}%
    \special{pa 749 2279}%
    \special{fp}%
    \special{pa 273 2000}%
    \special{pa 727 2000}%
    \special{fp}%
    \special{pa 197 939}%
    \special{pa 803 939}%
    \special{fp}%
    \special{pn 56}%
    \special{pa 197 939}%
    \special{pa 727 394}%
    \special{fp}%
    \special{pa 803 939}%
    \special{pa 273 545}%
    \special{fp}%
    \hbox{\vrule depth2.394in width0pt height 0pt}%
    \kern 1.000in
  }%
}%
}
\caption{$M_{n,r,t,b}$\label{Mfig}}
\end{center}
\end{figure}

Every vertex in $Y$ has degree exactly $r$ except the $r_1$ vertices of $S'$.
If we add no edges joining $S$ and $S'$, then every edge has an endpoint with
degree $r$, and each of the $r-1$ vertices of $S'$ has degree that exceeds $r$
by $n-2r$.  Hence $b$ equals $(r-1)(n-2r)$ plus the number of edges added
joining $S$ and $S'$.  The construction does not exist when $b<(r-1)(n-2r)$.

To form an $X$-matching containing an added edge joining $S$ and $S'$,
we must match $R$ into the remainder of $S'\cup \{u'\}$ (in $(r-1)!$ ways),
match $T$ into $T'$ (in one way), and match the remaining $r-1$ vertices of
$S\cup \{u\}$ into $R'$ (in $(r+t-1)!/t!$ ways).
Thus each edge joining $S$ and $S'$ (multi-edges allowed), adds exactly
$(r-1)!(r+t-1)!/t!$ $X$-matchings.  Also, the number of edges joining $S$ to
$S'$ is $b-(r-1)(n-2r)$.

Next suppose that we use no edge joining $S$ and $S'$ (whether such edges are
present or not).  Each such $X$-matching must match $R$ into $S'\cup \{u'\}$.
We can match $R$ into $S'$ in $(r-1)!$ ways.  This again forces $T$ matched
into $T'$.  However, now all $r$ vertices of $S\cup \{u\}$ remain to be matched
into $R'\cup\{u'\}$, which can be done in $(r+t-1)/t!$ ways using the edge
$uu'$ and $(r+t-1)!/(t-1)!$ ways not using $uu'$.  Note that
$(r-1)!(r+t-1)!/(t-1)!=(r-1)!(r+t-1)!t/t!$.

Finally, suppose that some vertex of $R$ is matched to $u'$, chosen in $r-1$
ways.  The rest of $R$ matches into $S'$ in $(r-1)!$ ways, leaving a vertex $y$
of $S'$ uncovered.  Without using $y$, we finish in $(r+t-1)!/(t-1)!$ ways,
matching $T$ into $T'$ and all $r$ vertices of $S\cup U$ into $R'$.
If we use $y$, then it matches to one vertex of $T$, chosen in $n-2r+1$ ways.
The mate of this vertex in $T'$ remains available.  To complete the
$X$-matching, we have $r+t-1$ choices for the mate of $u$ and then still
$(r+t-1)!/t!$ choices for covering the $r-1$ vertices of $S$.

Summing the counts in the various cases yields
\begin{align*}
\Phi(M_{n,r,t,b}) &=(r\!-\!1)!\FR{(r\!+\!t\!-\!1)!}{t!}
\big[b-(r\!-\!1)(n\!-\!2r)+1+t
+(r\!-\!1)\left\{t+(n\!-\!2r\!+\!1)(r\!+\!t\!-\!1)\right\} \big]\\
&=(r\!-\!1)!\FR{(r\!+\!t\!-\!1)!}{t!}\big[b+r(t+1)+(r-1)(n-2r+1)(r+t-2)\big]
\end{align*}
When $r=2$, the computation simplifies to $\Phi(M_{n,2,t,b})=(t+1)[b+2+(n-1)t]$.
\end{construction}

In Theorem~\ref{leafmain}, we will prove that $M_{n,r,t,b}$ minimizes
$\Phi$ for the parameters $n,t,b$ when $r=2$.  Furthermore, the
formula $(t+1)[b+2+(n-1)t]$ is the minimum value of $\Phi$ over the full
range of nonnegative $b$, not only $b\ge n-4$.

First we introduce another construction that covers the full range of $b$ when
$r=2$.  One may note that $M_{4,2,1,0}$ consists of an $8$-cycle plus an extra
vertex having the same neighbors as one vertex of $Y$; this is precisely
the graph $C_{4,1,0}$ in the construction below.

\begin{construction}\label{Cntb}
{\it An $X,Y$-bigraph $C_{n,t,b}$ with $\Phi(C_{n,t,b})=[(n-1)t+2+b](t+1)$.} 
Construct $C_{n,t,b}$ from the $2n$-vertex cycle $C_{2n}$ with $\C X=n$ by
adding $t$ copies of one vertex of $Y$ on the cycle (with the same two
neighbors in $X$) and adding $b$ copies of the edge on the cycle incident to
one of the vertices of $X$ in the copy $H$ of $K_{2,t+1}$.

Ignoring the extra copies of the multiedge, $C_{n,t,b}$ has $(n-1)(t+1)t$
$X$-matchings using two edges from $H$ (choose which one of the other $n-1$
vertices of $Y$ to leave uncovered).  When only one edge of $H$ is used, the
rest of the $X$-matching is determined, so this adds $2(t+1)$ $X$-matchings.
Among these, exactly $t+1$ matchings use the representative of the multiedge,
so each of the $b$ added copies adds $t+1$ $X$-matchings.
Hence $\Phi(C_{n,t,b})=[(n-1)t+2+b](t+1)$.

We mention several other special constructions with the right number of edges
and $X$-matchings.  (1) Merge a high-degree vertex of $K_{2,t+1}$ with one
vertex of a $(2n-2)$-cycle and add $b$ extra copies of an edge on the cycle
incident to the high-degree vertex.  (2) when $(n,t,b)=(4,2,0)$, let $G$
consist of three $4$-cycles with one common vertex (in $X$); here
$\Phi(G)=24=[(n-1)t+2+b](t+1)$.  (3) When $(n,t,b)=(3,2,0)$, let
$X=\{x_1,x_2,x_3\}$ with $x_1$ having two neigbhors, $x_2$ having three
neighbors (none in common with $x_1$), and $x_3$ adjacent to all five vertices
of $Y$ (so all vertices of $Y$ have degree $2$ and $b=0$).
Here $\Phi(G)=18=[(n-1)t+2+b](t+1)$.
\end{construction}

For large $n$, the graph $M_{n,2,t,b}$ is very different from $C_{n,t,b}$: the
latter has a long cycle, which the former does not, while the former has many
$4$-cycles, which the latter does not.  Nevertheless, we show that both
minimize $\Phi$ for given $n,t,b$.  Recall that ``leafless'' means that every
vertex has at least two neighbors: that is, $r=2$.  Multiedges are allowed.

\bigskip
\noindent
{\bf Theorem~\ref{leafmain}.}
{\it If $G$ is an $X$-surplus leafless $X,Y$-bigraph with $b+2(n+t)$ edges,
$\C X=n$, and $\C Y-\C X=t\ge0$, then
\begin{equation*}\label{2(t+1)}
\Phi(G)\geq [(n-1)t+2+b](t+1),
\end{equation*}
which is sharp in all cases.
}

\begin{proof}
Sharpness is shown by Construction~\ref{Cntb}.  For the lower bound, we use
induction on $m+n+t$.  Various base cases have been covered.
By Corollary~\ref{t+1cor}, we may assume that $G$ is a simple graph.
The condition that $G$ is leafless requires $n\ge2$.  When $n=2$, the only
simple leafless $X,Y$-bigraph is $K_{2,t+2}$, which has $(t+2)(t+1)$
$X$-matchings and $2(n+t)$ edges, so $b=0$ and $(n-1)t+2+b=t+2$, as desired.
When $t=0$, we have an elementary graph and have observed that
Theorem~\ref{surplust} provides the desired lower bound, since
the expression $[(n-1)t+2+b](t+1)$ reduces to $\C{E(G)}-\C{V(G)}+2$.

Hence we may assume $t\ge1$ and $n\ge3$ and that $G$ is simple and $X$-surplus.
In several steps, we reduce consideration of $G$ to a more restricted class.
A {\it pendant $4$-cycle} in a graph $G$ is a $4$-cycle containing one
cut-vertex and three vertices of degree $2$ in $G$.

\smallskip
{\bf Step 1:} {\it For any vertex of $X$ with degree $2$, both neighbors
have degree $2$, and all three vertices lie on a pendant $4$-cycle.}
Consider $x\in X$ having only two neighbors.  Let $X'=X-\{x\}$ and 
$N_G(x)=\{y,\hat y\}$.  Let $G^*=G-\{x,y,\hat y\}$ and $Y^*=Y-N(x)$.

We first prove Hall's Condiion for $G^*$.  If $\C{N_{G^*}(S)}<\C S$ for some
$S\esub X'$, then since $G$ is $X$-surplus and $\C{N_G(S)}>\C S$, we must have
$y,\hat y\in N_G(S)$, in which case $\C{N_G(S\cup\{x\})}=\C{S\cup\{x\})}$,
contradicting that $G$ is $X$-surplus.  Hence Hall's Condition holds for $G^*$.

Note that $n-1=\C{X'}\ge2$ and $\C{Y^*}=n+t-2$.  Add $t-1$ vertices to $X'$
that are adjacent via single edges to each vertex of $Y^*$, producing a graph
$H$.  Note that $H$ has $n+t-2$ vertices in each part, and each vertex of $Y^*$
has at least $t+1$ neighbors in $H$.  By the result of M.~Hall and the fact
that $G^*$ has an $X'$-matching, $H$ has at least $(t+1)!$ perfect matchings.
Each perfect matching in $H$ restricts to an $X'$-matching in $G^*$ when the
added vertices are deleted, and each $X'$-matching in $G^*$ arises $(t-1)!$
times as such a restriction.  Hence $\Phi(G^*)\ge (t+1)t$.

Let $G'$ be the $X',Y'$-bigraph obtained from $G$ by deleting $x$ and merging
$y$ and $\hat y$ into a new vertex $y'$, so $\C{Y'}-\C{X'}=\C Y-\C X=t$.
If $G'$ is leafless and $X'$-surplus, then we can apply the induction
hypothesis to it.  Note that $G'$ has two fewer edges than $G$ and omits a
vertex of $X$, so $(m-2)-2(n-1+t)=m-2(n+t)=b$.  Thus $G'$ being leafless and
$X'$-surplus yields $\Phi(G')\ge[(n-2)t+2+b](t+1)$.
Every $X'$-matching in $G'$ extends to an $X$-matching in $G$ by choosing a
neighbor for $x$.  Matchings that cover $y'$ extend in only one way, but
matchings that do not cover $y'$ extend in two ways, using either edge incident
to $x$.  The $X'$-matchings in $G'$ that do not use $y'$ are precisely the
$X'$-matchings in $G^*$.  We showed $\Phi(G^*)\ge t(t+1)$, so we gain at least
$t(t+1)$ in moving from $\Phi(G')$ to $\Phi(G)$, which yields
$\Phi(G)\ge [(n-1)t+2+b](t+1)$, as desired.

If $G'$ is not $X'$-surplus, then $\C{N_{G'}(S)}\le\C S$ for some
$S\esub X'$.  Since $G$ is $X$-surplus, $y'\in N_{G'}(S)$.  Now
$$
\C{N_G(S\cup\{x\})}=\C{N_{G'}(S)}+1\le\C S+1=\C{S\cup\{x\}},
$$
contradicting that $G$ is $X$-surplus.  Hence $G'$ is $X'$-surplus.

Thus if $G'$ is leafless then we have the desired number of $X$-matchings in
$G$, even with $x$ having degree $2$.  Suppose that  $G'$ is not leafless.
If $d_{G'}(x')=1$ for some $x'\in X'$, then $N_G(x')=\{y,\hat y\}$, but this
yields $N_G(\{x,x'\})=\{y,\hat y\}$, contradicting that $G$ is $X$-surplus.
The other possibility is $d_{G'}(y')=1$, requiring
$N_G(y)=N_G(\hat y)=\{x,x'\}$ for some $x'\in X'$.  Since $G$ is
$X$-surplus, $N_G(\{x,x'\})>2$, so $x'$ has another neighbor; call it $y''$.
Now the subgraph induced by $x,x',y,\hat y$ is a pendant $4$-cycle.

\smallskip
{\bf Step 2:} {\it If $t\ge3$, then every vertex of $X$ has at least three
neighbors.}
By Step 1, if $N(x)=\{y,\hat y\}$ for some $x\in X$, then $y$ and $\hat y$
have another common neighbor $x'$, and the subgraph induced by $x,x',y,\hat y$
is a pendant $4$-cycle.  Let $G^*=G-\{x,y,\hat y\}$, as in Step 1.  In $G^*$
with $N_G(y)=N_G(\hat y)=\{x,x'\}$, the only vertex of $X'$ having lost
neighbors from $G$ is $x'$.  Hence the graph $G^*+x'z$ obtained by adding the
edge $x'z$ to $G^*$ is leafless, where $z\in Y-\{y,\hat y,y''\}$.  This graph
has $m-3$ edges.  We deleted two vertices from $Y$ and one from $X$, so the
difference between the part-sizes is $t-1$.  In particular,
$(m-3)-2(n-1+t-1)=b+1$.  If $G^*+x'z$ is $X'$-surplus, then the induction
hypothesis yields
$\Phi(G^*+x'z)\ge [(n-2)(t-1)+3+b]t$.  We can obtain two $X$-matchings in $G$
for each $X'$-matching in $G^*+x'z$.  If such a matching omits $x'z$, then it
occurs in $G$ and we add $xy$ or $x\hat y$.  If the matching uses $x'z$, then
we replace that edge by $x'y$ or $x'\hat y$ and make $x$ adjacent to the
uncovered vertex in $\{y,\hat y\}$.  The resulting matchings are distinct,
yielding $\Phi(G)\ge [(n-2)(t-1)+3+b]2t$.

We can generate two more $X$-matchings in $G$.
Every $X$-matching in $G$ that we generated from an $X'$-matching in $G^*+x'z$
covered at most two members of $\{y,\hat y,z\}$; we find two more
that use all three of these vertices.  Let $G''=G-\{x,x'\}-N(x)$ and
$X''=X-\{x,x'\}$.  The graph $G''$ includes all neighbors in $G$ of vertices in
any subset of $X''$, so $G''$ is $X''$-surplus.  Although $G''$ need not be
leafless, the surplus condition implies that every edge of $G''$ lies in an
$X''$-matching.  In particular, $z$ is covered in some $X''$-matching in $G''$.
We can extend this to an $X$-matching in $G$ in two ways by matching $\{x,x'\}$
with $N(x)$.

It now suffices to have $[(n-2)(t-1)+3+b]2t+2\ge[(n-1)t+2+b](t+1)$.
By collecting like terms, this inequality can be rewritten as
$$(n-3)t(t-3)+b(t-1)\ge0.$$
Since $n\ge3$
and $b\ge0$, this is true in all cases with $t\ge3$.

Thus it suffices to prove that $G^*+x'z$ is $X'$-surplus for some
$z\in Y-\{y,\hat y,y''\}$.  We showed already that $G^*$ satisfies Hall's
Condition.  If $G^*$ is $X'$-surplus, then also $G^*+x'z$ is $X'$-surplus.
Hence we may consider the family $\CA$ of
subsets $S\esub X'$ such that $\C{N_{G^*}(S)}=\C S$; it is nonempty.  Since $G$
is $X$-surplus, each member of $\CA$ contains $x'$.  Since $G^*$ satisfies
Hall's Condition, $\CA$ is closed under union.  That is, submodularity of the
neighborhood function and Hall's Condition yield
$$
\C{N(S)}+\C{N(T)}\ge \C{N(S\cup T)}+\C{N(S\cap T)}
\ge\C{S\cup T}+\C{S\cap T}=\C S+\C T.
$$
Thus $S,T\in \CA$ implies $S\cup T,S\cap T\in\CA$.

We conclude that $\CA$ has a unique maximal member $A$.
If $A=X'$, then $t=1$, but we have eliminated that case.
Otherwise, we can choose $y$ outside $N_{G^*}(A)$.  Since $x'$ belongs to
all members of $\CA$, adding the edge $x'z$ enlarges the neighborhood of
each member of $\CA$, and then $G^*+x'z$ is $X'$-surplus.

\smallskip
{\bf Step 3:} {\it No subset of $X$ is slim, where a subset $S$ of $X$
is \emph{slim} if $\C{N(S)}=\C S+1$ and some member of $N(S)$ with at least
three neighbors has exactly one neighbor in $S$.}
Let $S$ be a smallest slim subset of $X$, if one exists.
Let $y'$ be a vertex of $N(S)$ with at least three neighbors such that $y'$
has only one neighbor $x'$ in $S$.  The existence of $y'$ requires
$n\ge\C S+2$.

If $\C S=1$, then $x'$ has two neighbors.  By Step 1, the neighbors of $x'$
have degree $2$.  Hence $y'$ cannot exist, and $S$ is not slim.

If $\C S\ge2$ and $N(S)$ has a vertex other than $y'$ with one neighbor in $S$,
then deleting its neighbor $x''$ in $S$ yields a subset of $S$ contradicting
the minimality of $S$ (if $x''\ne x'$) or contradicting that $G$ is $X$-surplus
(if $x''=x'$).

If $\C S=2$, then $n\ge4$ and $\C{N(S)}=3$ and the vertex $x$ of $S$ other than
$x'$ has neighbors only in $N(S)-\{y'\}$; let $N(x)=\{y,\hat y\}$.  Thus $x$
has only two neighbors, so by Step 2 the case $\C S=2$ occurs {\it only} when
$t\le2$.  By Step 1, $x$ and $N(x)$ lie on a pendant $4$-cycle.  Since $y$ and
$\hat y$ must have at least two neighbors in $S$, the fourth vertex of the
pendant $4$-cycle is $x'$.  Now let $G'=G-\{x,x',y,\hat y\}$.  Since $y'$ has
at least two neighbors outside $S$, the graph $G'$ is leafless.  Also $y$ and
$\hat y$ have neighbors only in $S$, so $G'$ is $X'$-surplus, where $X'=X-S$.

Note that $G'$ has $m-5$ edges, and $\C{X'}=n-2$, and $\C{Y-N(x)}=\C Y-2$.
We compute $(m-5)-2(n-2+t) = m-1-2(n+t) = b-1$.  By the induction hypothesis,
$\Phi(G')\ge[(n-3)t+1+b](t+1)$.  Every $X'$-matching in $G'$ extends to an
$X$-matching in $G$ in two ways.  None of the resulting matchings use the edge
$x'y'$.  To generate such matchings, note that $G'-y'$ satisfies Hall's
Condition, since $G'$ is $X'$-surplus.  Also $G'-y'$ has no isolated vertices,
since the only neighbors in $G$ that its vertices may have lost are $x'$ and
$y'$.  Hence by Lemma~\ref{t+1} $\Phi(G'-y')\ge t$.  We obtain $2t$ additional
$X$-matchings in $G$ by adding $x'y'$ and matching $x$ to $y$ or $\hat y$.

This gives us
$\Phi(G)\ge[(n-3)t+1+b]2(t+1)+2t$, and we need
$$
[(n-3)t+1+b]2(t+1)+2t\ge[(n-1)t+2+b](t+1).
$$
We have observed that this case occurs only when $t\le2$ and $n\ge4$.
When $t=2$ the needed inequality simplifies to
$[2n-5+b]6+4\ge[2n+b]3$ and then $6n+3b\ge26$, which is valid unless
$(n,b)=(4,0)$.  Fortunately, the case $b=0$ requires every vertex in $Y$ to
have degree $2$, which immediately forbids slim sets.
When $t=1$, the inequality simplifies to
$[n-2+b]4+2\ge [n+1+b]2$ and then $2n+2b\ge8$, which again holds since
$n\ge4$.

Therefore, we may assume $\C S\ge3$.
Switching notation, let $x''$ be the unique neighbor of $y'$ in $S$.  Since
$d(y')\ge3$, by Step 1 also $d(x'')\ge3$.  If $S$ has a vertex $x$ of
degree $2$, then by Step 1 it lies on a pendant $4$-cycle.  Let $x'$ be the
other vertex of $X$ in the $4$-cycle; since vertices in $N(x)$ must have two
neighbors in $S$, also $x'\in S$.  If $x'\ne x''$, then
$N(S-\{x,x'\})=N(S)-\{y,\hat y\}$, and $S-\{x,x'\}$ is a smaller slim set,
with $y'$ still serving as the special vertex.  If $x'=x''$, then
$N(S-\{x,x'\})=N(S)-\{y,\hat y,y'\}$, contradicting that $G$ is $X$-surplus.
Hence the minimality of $S$ forbids vertices of degree $2$ from $S$.

Now, with degree at least $3$ at each vertex of $S$, at least $3\C S-1$ edges
join $S$ to $N(S)-\{y'\}$.  Since $3\C S-1>2\C S$ when
$\C S\ge2$, some $y_1\in N(S)-\{y'\}$ has at least $3$ neighbors in $S$; call
them $x_1,x_2,x_3$.  We may assume $x_1\ne x$.

Let $\hG=G-x_1y_1$.  Since $x_1$ and $y_1$ both have degree at least $3$ in
$G$, the graph $\hG$ is leafless.  Also $\hG$ has $m-1$ edges.
If $\hG$ is $X$-surplus, then the induction hypothesis provides
$\Phi(\hG)\ge [(n-1)t+1+b](t+1)$, and by Corollary~\ref{t+1cor} $t+1$ more
$X$-matchings use $x_1y_1$.
 
Since $\hG$ is missing only one edge from $G$, Hall's Condition holds for
$X$ in $\hG$.  If $\hG$ is not $X$-surplus, then there is a set $S'\esub X$
such that $\C{N_{\hG}(S')}=\C{S'}$.  Since $\hG$ lacks only $x_1y_1$ and
$G$ is $X$-surplus, we have $x_1\in S'$ and $\C{N_G(S')}=\C{S'}+1$, with
$x_1$ being the only neighbor of $y_1$ in $S'$.  Thus $S'$ is slim in $G$.
Also, since $x_2,x_3\in N(y_1)$, we conclude $x_2,x_3\notin S'$.

We apply submodularity of the neighborhood function in $G$ on subsets of
$X$, plus $G$ being $X$-surplus and $S,S'$ being slim:
\begin{align*}
\C{S\cup S'}+\C{S\cap S'}+2
&\le \C{N_G(S\cup S')}+\C{N_G(S\cap S')}
\le \C{N_G(S)}+\C{N_G(S')}\\
&= \C S+1+\C{S'}+1
=\C{S\cup S'}+\C{S\cap S'}+2.
\end{align*}
We conclude that equality holds throughout.  Hence
$\C{N(S\cap S')}=\C{S\cap S'}+1$.  However, $x_1\in S\cap S'$ and 
$x_2,x_3\in S-S'$, with $x_1$ the only neighbor of $y_2$ in $S\cap S'$.
Thus $S\cap S'$ is a smaller slim set than $S$, contradicting the choice of
$S$.  Thus there is no slim set.

\smallskip
{\bf Step 4:} {\it Every vertex of $Y$ has exactly two neighbors.}
Suppose $x_1,x_2,x_3\in N_G(y)$ for some $y\in Y$.  By Step 1, the neighbors of
a vertex of $X$ having degree $2$ must also have degree $2$, so all neighbors
of $y$ have degree at least $3$.  Let $G'=G-x_1y$.  Since $d_G(x_1)\ge3$, the
graph $G'$ is leafless.  If $G'$ is $X$-surplus, then the induction hypothesis
yields $\Phi(G')\ge[(n-1)t+1+b](t+1)$, and Corollary~\ref{t+1cor} yields at
least $t+1$ more $X$-matchings using the edge $x_1y$.

Hence it suffices to show that $G'$ is $X$-surplus; suppose not.  Consider a
nonempty $S\esub X$ with $\C{N_{G'}(S)}=\C S$.  Since $G$ is $X$-surplus, we
must have $x_1\in S$, and $x_1$ is the only neighbor of $y$ in $S$, so
$\C{N_G(S)}=\C S+1$.  Now $S$ is a slim set in $G$, which by Step 2 does not
exist.  Hence we may restrict to the case where every vertex of $Y$ has
exactly two neighbors.

\nobreak
\smallskip
{\bf Step 5:} {\it $\Phi(G)\ge[(n-1)t+2](t+1)$ when $t\ge3$.}
Since every vertex of $Y$ has degree $2$, we have $m=2(n+t)$, so $b=0$.

Suppose first that $X$ contains a subset $S$ with $\C{N(S)}=\C S+1$; let $S$ be
a smallest such set.  If some $y\in N(S)$ has only one neighbor $x$ in $S$,
then $S-\{x\}$ contradicts the minimality of $S$.  Hence every vertex of $N(S)$
has both of its neighbors in $S$.  By Step 2, $\delta_X(G)\ge3$, so the number
of edges joining $X$ and $N(S)$ is at least $3\C S$, but it also equals
$2(\C S+1)$.  Hence $\C S\le 2$, and $\delta_X(G)\ge3$ forces equality.

Now $G$ has $K_{2,3}$ as a component, with six $S$-matchings.  Deleting this
component yields a leafless graph with surplus.  Also $b=0$, since each vertex
of $Y$ has degree $2$.  Hence it suffices to have
$6[(n-3)(t-1)+2]t\ge [(n-1)t+2](t+1)$.
When $K_{2,3}$ is a component, $G$ being leafless forces $n\ge4$, and then the
desired inequality holds.

Hence we may assume $\C{N(S)}\ge \C S+2$ for all nonempty $S\esub X$.  When we
delete any $y\in Y$, the graph $G-y$ is leafless (since vertices of $X$ have at
least three neighbors) and is $X$-surplus (since no vertex other than $y$ is
lost from any neighborhood).  Thus $\Phi(G-y)\ge [(n-1)(t-1)+2]t$.  Summing
this inequality over all $y\in Y$ counts each $X$-matching in $G$ exactly $t$
times, once for each vertex of $Y$ it does not cover.  Hence
\begin{align*}
\Phi(G)&\ge[(n-1)(t-1)+2]\FR{t(n+t)}t =[(n-1)(t-1)+2](t+1)+[(n-1)(t-1)+2](n-1)\\
&=[(n-1)(t-1)+2](t+1)+(n-1)(t+1)+[(n-1)(t-1)-(t+1)+2](n-1)\\
&=[(n-1)t+2](t+1)+(n-1)(t-1)(n-2)>[(n-1)t+2](t+1).
\end{align*}

\smallskip
{\bf Step 6:} {\it $\Phi(G)\ge[(n-1)t+2](t+1)$ when $t\le2$.}
Since every vertex of $Y$ has degree $2$, we have $m=2(n+t)$, so $b=0$.

Suppose first that some $y\in Y$ is the only common neighbor of its neighbors
$x_1$ and $x_2$.  In this case obtain $G'$ from $G-y$ by merging $x_1$ and 
$x_2$ into a vertex $x'$.  Let $X'$ be the set obtained from $X$ by merging
$x_1$ and $x_2$.  Since $x_1$ and $x_2$ both have a neighbor other than $y$,
and those neighbors are distinct, $G'$ is leafless.  In moving from $G$ to
$G'$, any subset of $X-\{x_1,x_2\}$ has as many neighbors as before, and
subsets of $X'$ containing the merged vertex may have more neigbors than the
corresponding subsets using just $x_1$ or $x_2$.  Hence $G'$ is $X'$-surplus.
Vertices of $Y-\{y\}$ all have degree $2$ in $G'$.
By the induction hypothesis, $\Phi(G')\ge[(n-2)t+2](t+1)$.  Since the edges
incident to $x'$ in $G'$ do not cover $y$ in $G$, each $X'$-matching in $G'$
extends to an $X$-matching in $G$ by adding $x_1y$ or $x_2y$.

We need $t(t+1)$ more $X$-matchings in $G$.  Those that we have found cover
$y$, we find additional $X$-matchings that do not cover $y$.  Since $G$ is
$X$-surplus, the graph $G-y$ satisfies Hall's Condition and contains an
$X$-matching $M$.  In the subgraph of $G$ induced by the vertices of $M$, every
vertex belonging to $Y$ has at least two neighbors, and $M$ is a perfect
matching.  Reversing the roles of $X$ and $Y$ in applying Theorem~\ref{main},
we find at least two perfect matchings in this subgraph, say $M$ and $M'$;
these are $X$-matchings in $G$.

When $t=1$, we only need two $X$-matchings in $G$ that do not cover $y$, so
$M$ and $M'$ suffice.  When $t=2$, we need six $X$-matchings in $G$ that do not
cover $y$.  The matchings $M$ and $M'$ both omit $y$ and another vertex
$y'\in Y$.  To find four more $X$-matchings, we find two $X$-matchings in $G-y$
using each of the two edges incident to $y'$.  Let $x'y'$ be one such edge.
Switching the edges covering $x'$ in $M$ and $M'$ to $x'y'$ instead yields two
matchings.  The switches cannot produce the same matching, because $M$ and $M'$
are perfect matchings in $G-\{y,y'\}$, and if they agree on the edges covering
other vertices of $X$ then they also agree on the edge covering $x'$.  These
matchings using $x'y'$ are also different from the two resulting $X$-matchings
in $G-y$ using the other edge at $y'$ that are obtained in the same way.

In the remaining case, any two vertices in $X$ having a common neighbor have at
least two common neighbors.  If $G$ is disconnected, then there can only be two
components, each having one more vertex in $Y$ than in $X$.  With $k$ vertices
of $X$ in one component and $n-k$ in the other, where $2\le k\le n-2$ since $G$
is leafless, we have
$$\Phi(G)\ge[(k-1)+2]2[(n-k-1)+2]2\ge12n-12\ge6n=[(n-1)2+2]3.$$

We may therefore assume that $G$ is connected.  Since every vertex of $Y$ has
degree $2$, at least $n-1$ pairs of vertices in $X$ must each have at least
two common neighbors.  Thus $2(n+t)=\C{E(G)}\ge 4(n-1)$, which simplifies to
$n\le t+2$.  We also have restricted to $n\ge3$, so the remaining cases are
$(n,t)$ being $(3,1)$, $(3,2)$, and $(4,2)$.

When $n=3$, if there is a vertex $x\in X$ having degree $2$, then by Step 1
$x$ lies on a pendant $4$-cycle with its neighbors $y$ and $\hat y$ and the
cut-vertex $x'$.  Since $x$ has no further neighbors, all remaining vertices
of $Y$ are adjacent to $x'$ and the third vertex $x^*$ of $X$.  When $t=1$, the 
graph consists of two $4$-cycles sharing $x'$; it has eight $X$-matchings,
which equals the desired lower bound, providing a sharpness example.  When
$t=2$, the vertices $x$ and $x^*$ have degrees $2$ and $3$, respectively,
with no common neighbors, and $x'$ is adjacent to all five vertices of $Y$.
There are $18$ $X$-matchings, again a sharpness example.

Hence when $n=3$ we may assume that all three vertices of $X$ have at least
three neighbors.  This cannot happen when $t=1$, since then the graph has only
eight edges.  When $t=2$, each of the five vertices in $Y$ is a common neighbor
for two vertices of $X$.  Let $a$, $b$, and $c$ be the numbers of common
neighbors for the three pairs of vertices in $X$.  We have $a+b+c=5$ and any
two of these sum to at least $3$, since the sum is the degree of a vertex in
$X$.  By symmetry, we may assume $(a,b,c)=(1,2,2)$.  This determines $G$, and
explicit counting yields $\Phi(G)=20$, while the desired lower bound is only
$18$.

Only the case $(n,t)=(4,2)$ remains, achieving equality in $2(n+t)\ge 4(n-1)$.
Hence exactly three pairs of vertices in $X$ have common neighbors in $Y$,
each occurring exactly twice.  To keep $G$ connected, those three pairs may
form a star or a path on $X$, and we need at least $24$ $X$-matchings.  In the
case of a star, $G$ consists of three $4$-cycles with one common vertex, and
the number of $X$-matchings is exactly $24$, providing another isolated
sharpness example.  In the case of a path, $G$ is a chain of three
edge-disjoint $4$-cycles merged at vertices of $X$, and there are $28$
$X$-matchings.
\end{proof}

Due to the wide variety of extremal examples in Constructions~\ref{linearn}
and~\ref{Cntb}, we do not expect a nice characterization of the $X,Y$-bigraphs
achieving equality in Theorem~\ref{leafmain}.  Nevertheless, we can combine
Theorem~\ref{leafmain} with Corollary~\ref{maincor} to describe a general lower
bound on $\Phi(G)$ that includes instances when $G$ is not $X$-surplus.

\begin{corollary}
Fix integers $n,t,r$ greater than $1$.
Let $G$ be an $X,Y$-bigraph with $\C X=n$ and $\C Y=n+t\ge n$,
such that $\delta_X\ge k$ and $\C{N_G(x)}\ge r$ for all $x\in X$.
Let $p=n-\alpha'(G)$.  If $n'=n-\C S$, where $S$ be a largest subset of $X$
such that $\C{N(S)}=\C S-p$, then
$$
\Phi(G)\ge (k-r+1)\FR{(r+p)!}{p!}\big[(n'-1)(t+p)+2+b'\big](t+p+1),
$$
where $b'+2(n'+t+p)$ is the number of edges in the subgraph $G'$
obtained by deleting $S\cup N(S)$.
\end{corollary}
\begin{proof}
By Theorem~\ref{main} and Corollary~\ref{maincor}, the number of 
maximum matchings of the subgraph induced by $S\cup N(S)$ is at least
$(k-r+1)(r+p)!/p!$.  Each such matching can be paired with an
$(X-S)$-matching of $G'$ to obtain a maximum matching in $G$.

The graph $G'$ is leafless and $X$-surplus, by the choice of $S$.
It has $n'$ vertices in $X$ and $n'+t+p$ vertices in $Y$.  By
Theorem~\ref{leafmain}, $G'$ has at least $[(n'-1)(t+p)+2+b'](t+p+1)$
$(X-S)$-matchings.
\end{proof}

The most interesting question that remains from our study is the following

\begin{question}
What is the minimum of $\Phi(G)$ when $G$ is an $X$-surplus $X,Y$-bigraph
in which $\C X=n$ and $\C Y=n+t$ and every vertex has at least $r$ neighbors?
\end{question}

\section{Restrictions on $Y$}\label{sec:restrict}

Setting $r=1$ in Theorem~\ref{main} is equivalent to eliminating the
restriction on $r$.  Theorem~\ref{main} then only guarantees $\Phi(G)\ge k$
when $G$ is an $X,Y$-bigraph with $\Phi(G)\ge1$ and $\delta_X(G)\ge k$, which
is sharp by Construction~\ref{sharp1}.  However, equality in this construction
requires $\C Y=\C X$ and $\delta_Y=1$.  Forbidding this situation permits
a better lower bound.  We begin with several constructions.

\begin{construction}\label{sharp2}
For $k=2$, the number of $X$-matchings in $K_{1,2}$ or in any even cycle is
exactly $2$, which equals $2k-2$.  The further constructions below are 
illustrated in Figure~\ref{figG6}.

For $\C X=2$, define $F_k$ from a copy of $K_{2,2}$ with $X=\{x_1,x_2\}$
and $Y=\{y_1,y_2\}$ by adding $k-2$ copies of $x_1y_1$ and $k-2$ copies of
$x_2y_1$.  Now $\C X=\C Y=2$, $\delta_X(F_k)=k$, $\delta_Y(F_k)=2$,
and each edge incident to $y_1$ determines one perfect matching, so there are
$2k-2$ perfect matchings.  This is the special case of
Construction~\ref{sharp1} for $n=r=2$.

For $k=3$ and $\C X = \C Y=3$, we construct an $X,Y$-bigraph $G_6$.
Beginning with a $6$-cycle, fix a vertex $y\in Y$ and add one edge joining $y$
to each vertex of $X$ (two resulting edges have multiplicity 2).  Note
$\delta_X(G_6)=3$ and $\delta_Y(G_6)=2$.  Each edge incident to $y$ appears in
exactly one perfect matching, making five $X$-matchings in total, equal to
$2k-1$.

\begin{centering}
\begin{figure}[h]
\gpic{
\expandafter\ifx\csname graph\endcsname\relax \csname newbox\endcsname\graph\fi
\expandafter\ifx\csname graphtemp\endcsname\relax \csname newdimen\endcsname\graphtemp\fi
\setbox\graph=\vtop{\vskip 0pt\hbox{%
    \graphtemp=.5ex\advance\graphtemp by 0.656in
    \rlap{\kern 0.164in\lower\graphtemp\hbox to 0pt{\hss $\bu$\hss}}%
    \graphtemp=.5ex\advance\graphtemp by 0.383in
    \rlap{\kern 0.164in\lower\graphtemp\hbox to 0pt{\hss $\bu$\hss}}%
    \graphtemp=.5ex\advance\graphtemp by 0.656in
    \rlap{\kern 0.984in\lower\graphtemp\hbox to 0pt{\hss $\bu$\hss}}%
    \graphtemp=.5ex\advance\graphtemp by 0.383in
    \rlap{\kern 0.984in\lower\graphtemp\hbox to 0pt{\hss $\bu$\hss}}%
    \special{pn 11}%
    \special{pa 164 656}%
    \special{pa 984 383}%
    \special{fp}%
    \special{pn 8}%
    \special{ar -59 -1380 2049 2049 1.036634 1.461458}%
    \special{ar 1207 2419 2049 2049 -2.104959 -1.680135}%
    \special{pn 11}%
    \special{pa 164 383}%
    \special{pa 984 383}%
    \special{fp}%
    \special{pn 8}%
    \special{ar 574 -1345 1776 1776 1.337928 1.803665}%
    \special{ar 574 2111 1776 1776 -1.803665 -1.337928}%
    \special{pn 11}%
    \special{pa 164 656}%
    \special{pa 984 656}%
    \special{fp}%
    \special{pa 984 656}%
    \special{pa 164 383}%
    \special{fp}%
    \graphtemp=.5ex\advance\graphtemp by 0.656in
    \rlap{\kern 0.000in\lower\graphtemp\hbox to 0pt{\hss $x_2$\hss}}%
    \graphtemp=.5ex\advance\graphtemp by 0.383in
    \rlap{\kern 0.000in\lower\graphtemp\hbox to 0pt{\hss $x_1$\hss}}%
    \graphtemp=.5ex\advance\graphtemp by 0.656in
    \rlap{\kern 1.175in\lower\graphtemp\hbox to 0pt{\hss $y_2$\hss}}%
    \graphtemp=.5ex\advance\graphtemp by 0.383in
    \rlap{\kern 1.175in\lower\graphtemp\hbox to 0pt{\hss $y_1$\hss}}%
    \graphtemp=.5ex\advance\graphtemp by 0.792in
    \rlap{\kern 2.077in\lower\graphtemp\hbox to 0pt{\hss $\bu$\hss}}%
    \graphtemp=.5ex\advance\graphtemp by 0.519in
    \rlap{\kern 2.077in\lower\graphtemp\hbox to 0pt{\hss $\bu$\hss}}%
    \graphtemp=.5ex\advance\graphtemp by 0.246in
    \rlap{\kern 2.077in\lower\graphtemp\hbox to 0pt{\hss $\bu$\hss}}%
    \graphtemp=.5ex\advance\graphtemp by 0.792in
    \rlap{\kern 2.896in\lower\graphtemp\hbox to 0pt{\hss $\bu$\hss}}%
    \graphtemp=.5ex\advance\graphtemp by 0.519in
    \rlap{\kern 2.896in\lower\graphtemp\hbox to 0pt{\hss $\bu$\hss}}%
    \graphtemp=.5ex\advance\graphtemp by 0.246in
    \rlap{\kern 2.896in\lower\graphtemp\hbox to 0pt{\hss $\bu$\hss}}%
    \special{pa 2077 792}%
    \special{pa 2896 792}%
    \special{fp}%
    \special{pa 2896 792}%
    \special{pa 2077 519}%
    \special{fp}%
    \special{pa 2077 519}%
    \special{pa 2896 246}%
    \special{fp}%
    \special{pa 2896 246}%
    \special{pa 2077 246}%
    \special{fp}%
    \special{pa 2077 519}%
    \special{pa 2896 519}%
    \special{fp}%
    \special{pn 8}%
    \special{ar 1853 -1244 2049 2049 1.036634 1.461458}%
    \special{ar 3120 2556 2049 2049 -2.104959 -1.680135}%
    \special{ar 3120 -1517 2049 2049 1.680135 2.104959}%
    \special{ar 1853 2283 2049 2049 -1.461458 -1.036634}%
    \graphtemp=.5ex\advance\graphtemp by 0.929in
    \rlap{\kern 3.989in\lower\graphtemp\hbox to 0pt{\hss $\bu$\hss}}%
    \graphtemp=.5ex\advance\graphtemp by 0.656in
    \rlap{\kern 3.989in\lower\graphtemp\hbox to 0pt{\hss $\bu$\hss}}%
    \graphtemp=.5ex\advance\graphtemp by 0.383in
    \rlap{\kern 3.989in\lower\graphtemp\hbox to 0pt{\hss $\bu$\hss}}%
    \graphtemp=.5ex\advance\graphtemp by 0.109in
    \rlap{\kern 3.989in\lower\graphtemp\hbox to 0pt{\hss $\bu$\hss}}%
    \graphtemp=.5ex\advance\graphtemp by 0.929in
    \rlap{\kern 4.809in\lower\graphtemp\hbox to 0pt{\hss $\bu$\hss}}%
    \graphtemp=.5ex\advance\graphtemp by 0.656in
    \rlap{\kern 4.809in\lower\graphtemp\hbox to 0pt{\hss $\bu$\hss}}%
    \graphtemp=.5ex\advance\graphtemp by 0.383in
    \rlap{\kern 4.809in\lower\graphtemp\hbox to 0pt{\hss $\bu$\hss}}%
    \graphtemp=.5ex\advance\graphtemp by 0.109in
    \rlap{\kern 4.809in\lower\graphtemp\hbox to 0pt{\hss $\bu$\hss}}%
    \special{ar 1530 -2349 4098 4098 0.643501 0.927295}%
    \special{ar 7268 3388 4098 4098 -2.498092 -2.214297}%
    \special{ar 2908 -1853 2732 2732 0.801525 1.164062}%
    \special{ar 5890 2619 2732 2732 -2.340068 -1.977530}%
    \special{ar 3765 -1653 2049 2049 1.036634 1.461458}%
    \special{ar 5032 2146 2049 2049 -2.104959 -1.680135}%
    \special{pn 11}%
    \special{pa 3989 109}%
    \special{pa 4809 109}%
    \special{fp}%
    \special{pn 8}%
    \special{ar 4399 -1618 1776 1776 1.337928 1.803665}%
    \special{ar 4399 1837 1776 1776 -1.803665 -1.337928}%
    \special{pn 11}%
    \special{pa 3989 383}%
    \special{pa 4809 383}%
    \special{fp}%
    \special{pa 4809 383}%
    \special{pa 3989 929}%
    \special{fp}%
    \special{pa 3989 929}%
    \special{pa 4809 656}%
    \special{fp}%
    \special{pa 4809 656}%
    \special{pa 3989 656}%
    \special{fp}%
    \special{pn 8}%
    \special{ar 4399 -798 1776 1776 1.337928 1.803665}%
    \special{ar 4399 2657 1776 1776 -1.803665 -1.337928}%
    \graphtemp=.5ex\advance\graphtemp by 0.929in
    \rlap{\kern 3.825in\lower\graphtemp\hbox to 0pt{\hss $x_n$\hss}}%
    \graphtemp=.5ex\advance\graphtemp by 0.109in
    \rlap{\kern 3.825in\lower\graphtemp\hbox to 0pt{\hss $x_1$\hss}}%
    \graphtemp=.5ex\advance\graphtemp by 0.929in
    \rlap{\kern 5.000in\lower\graphtemp\hbox to 0pt{\hss $y_n$\hss}}%
    \graphtemp=.5ex\advance\graphtemp by 0.109in
    \rlap{\kern 5.000in\lower\graphtemp\hbox to 0pt{\hss $y_1$\hss}}%
    \graphtemp=.5ex\advance\graphtemp by 1.202in
    \rlap{\kern 0.574in\lower\graphtemp\hbox to 0pt{\hss $F_4$\hss}}%
    \graphtemp=.5ex\advance\graphtemp by 1.202in
    \rlap{\kern 2.486in\lower\graphtemp\hbox to 0pt{\hss $G_6$\hss}}%
    \graphtemp=.5ex\advance\graphtemp by 1.202in
    \rlap{\kern 4.399in\lower\graphtemp\hbox to 0pt{\hss $H'_{4,3}$\hss}}%
    \hbox{\vrule depth1.311in width0pt height 0pt}%
    \kern 5.000in
  }%
}%
}

\vspace{-1pc}
\caption{Construction~\ref{sharp2}. \label{figG6}}
\end{figure}
\end{centering}

Let $H_{n,k}$ be the $X,Y$-bigraph with $X=\{\VEC x1n\}$ and $Y=\{\VEC y1n\}$
whose edge set consists of one copy of $x_iy_i$ for $1\le i\le n$ plus
$k-1$ additional copies of $x_iy_1$ for $1\le i\le n$ (so $x_1y_1$ has 
multiplicity $k$).  Note $\delta_X(H_{n,k})=k$ and $\delta_Y(H_{n,k})=1$.
An $X$-matching must pair each $x_i$ with $y_i$.  Since $x_1y_1$ has
multiplicity $k$, there are $k$ $X$-matchings.

Form $H'_{n,k}$ from $H_{n,k}$ by adding one edge from $x_n$ to each of
$\VEC y2n$.  Now $\delta_X(H'_{n,k})=k$ and $\delta_Y(H'_{n,k})=2$.  Since
$N(y_n)=\{x_n\}$, still $y_n$ can only match to $x_n$, and then $x_iy_i$ for
$2\le i\le n$ are also forced.  Since $x_ny_n$ has multiplicity $2$, we have
$\Phi(H_{n,k})=2k$.

Another construction with $2k$ $X$-matchings satisfies the stronger condition
that each vertex of $Y$ has at least two distinct neighbors, without additional
$X$-matchings.  Starting with $H_{n,k}$, choose $i$ with $2\le i<n$ and add
another matching of $\VEC xin$ into $\VEC yin$ to create a cycle $C$ of length
$2(n-i+1)$.  Also add edges from $x_n$ to each of $\VEC y2{i-1}$.  An
$X$-matching must use one of the two matchings on $C$ and one of the $k$ copies
of $x_1y_1$.
\end{construction}

We prove a stronger form of Theorem~\ref{Y2}, characterizing extremality
in some cases.  The characterization of $\Phi(G)=2k-2$ here is used when $k=2$
in Case 2 of Theorem~\ref{nky2}.

\begin{theorem}\label{Y2+}
Let $G$ be an $X,Y$-bigraph with $\delta_X(G)\ge k$ and $\delta_Y(G)\ge1$
satisfying Hall's Condition and $\C X\ge2$.  If $\C Y>\C X$ or
$\delta_Y(G)\ge2$, then 
$$
\Phi(G)\ge\begin{cases}
   2k-2, &\text{if $\C X=2$ or $k=2$ (sharp only for $F_k$ and even cycles);}\\
   2k-1, &\text{if $\C X>2$ and $k=3$ (sharp only for $\C X=\C Y=3$ with $G=G_6$);}\\
   2k,   &\text{if $\C X\ge3$ and $k\ge4$ (sharp in all cases).}
\end{cases}
$$
\end{theorem}

\vspace{-1pc}
\begin{proof}
The sharpness examples are in Construction~\ref{sharp2}.  For the lower bounds,
we use induction on $\C X+\C Y$.  Suppose first that $\C X=\C Y=2$
and $\delta_Y(G)\ge2$.  If some $x\in X$ has only one neighbor, then it has
multiplicity at least $k$ to its neighbor $y\in Y$, and the other vertex
$y'$ in $Y$ must have multiplicity at least $2$ to the other vertex $x'$ in $X$.
In this case, $\Phi(G)\ge 2k$.  Otherwise, each vertex of $X$ is adjacent to
each vertex of $Y$, and Theorem~\ref{main} with $r=2$ yields $\Phi(G)\ge 2k-2$.
The description of equality for $r=n=2$ in Theorem~\ref{maineq} allows equality
here only for $F_k$.  In other cases with $\C Y>\C X=2$ or $k=2$ and $\C X\ge3$
we will show $\Phi(G)>2k-2$ (except equality for
even cycles).

For the induction step, $\C X+\C Y\ge 5$.
Since Hall's Condition holds, $\C Y\ge \C X$.

{\bf Case 1:} {\it $\C{N(S)}=\C S$ for some nonempty proper subset $S$ of $X$.}
Let $G_1$ be the subgraph of $G$ induced by $S\cup N(S)$, and let 
$G_2=G-V(G_1)$.  By Theorem~\ref{main} with $r=1$, $G_1$ has at least
$k$ $S$-matchings.

If $\C Y=\C X$, then $\delta_Y(G)\ge2$, by hypothesis.  In this case, we apply
Theorem~\ref{main} with $k=2$ and $r=1$ to $G_2$ as a $Y',X'$-bigraph,
where $Y'=Y-N(S)$ and $X'=X-S$, to obtain at least two perfect matchings in
$G_2$, each of which combines with each $S$-matching of $G_1$ to form an
$X$-matching of $G$.  This yields at least $\Phi(G)\ge 2k$.

If $\C Y>\C X$, then let $S'=X-S$.  Any $X$-matching of $G$ restricts to an
$S'$-matching $M$ in $G_2$ that omits some vertex $y\in Y-N(S)$.  Since $G$ has
no isolated vertex, $y$ has a neighbor $x$ in $S'$.  Replacing the edge
covering $x$ in $M$ with $xy$ yields another $S'$-matching in $G_2$.  Each
$S'$-matching in $G_2$ extends any $S$-matching in $G_1$, yielding
$\Phi(G)\ge2k$.

{\bf Case 2:} {\it $\C{N(S)}>\C S$ for every nonempty proper subset $S$ of $X$.}
Here setting $\C S=1$ implies that each vertex in $X$ has at least two
neighbors.
Deleting the endpoints of any edge preserves Hall's Condition,
so every edge appears in an $X$-matching.

If $k=2$ and some vertex of $X$ has degree at least $3$, then since every edge
appears in an $X$-matching we already have strict inequality in the bound.
Hence we may assume that every vertex of $X$ has degree $2$.  Now if $G$ has
more than one component or is a path with length more than $2$, then we again
have extra $X$-matchings beyond $2k-2$.  This leaves only even cycles,
which have exactly two $X$-matchings.

Now consider $k\ge3$, and suppose that $G$ does not contain the desired number
of $X$-matchings.  We claim first that every $x\in X$ has a neighbor $y$ such
that $x'y$ has multiplicity $d(x')-1$ for some $x'\in N(y)$.  If $x$ has no
such neighbor, then for every $y\in N(x)$ the graph $G-\{x,y\}$ has an
$X$-matching, and each vertex $x'$ of $X-\{x\}$ has degree at least $2$ in
$G-\{x,y\}$ (since $x'y$ has multiplicity at most $d(x')-2$).  By
Theorem~\ref{main} with $r=1$, the graph $G-\{x,y\}$ has at least two
$(X-\{x\})$-matchings.  Each edge incident to $x$ now appears in at least two
$X$-matchings in $G$, yielding $\Phi(G)\ge2k$.  This proves the claim.

Now for some $x\in X$, take $y$ and $x'$ as provided by the claim.
Since $x'y$ has multiplicity $d(x')-1$, one more edge $x'y'$ is incident
to $x'$.  Form $G'$ from $G-x'$ by merging $y$ and $y'$ into a new vertex
$\hat y$ inheriting the edges of $G-x'$ incident to both $y$ and $y'$.  Now
$G'$ is an $X',Y'$-bigraph, where $X'=X-\{x'\}$.  In moving from $G$ to $G'$
any subset of $X$ that remains loses at most one neighbor, so Hall's Condition
holds for $G'$.  Hence $G'$ has an $X'$-matching.  Also $\delta_{X'}(G')\ge k$,
since only edges incident to $x'$ were discarded.  Hence we will be able
to apply the induction hypothesis to $G'$ if $\C X>2$.

If $\C X=2$, then let $q$ be the number of edges incident to $x$ but not $y$.
We have at least $k-q$ $X$-matchings using $x'y'$ and a copy of $xy$.  We have
at least $(k-1)q$ $X$-matchings using a copy of $x'y$ and an edge at $x$ not 
incident to $y$.  Hence $\Phi(G)\ge k-q+(k-1)q=k+q(k-2)$.  Since $q\ge1$,
we have $\Phi(G)\ge2k-2$, and equality requires $q=1$.  If $\C Y>\C X$, then we
obtain an additional $X$-matching not covering $y$.  Hence equality requires
$\C Y=\C X$ and $G=F_k$.  Hence we may assume $\C X>2$.

If $\C Y>\C X$, then $\C{Y'}>\C{X'}$, since $\C Y-\C{Y'}=\C X-\C{X'}=1$.  On
the other hand, if $\delta_Y(G)\ge2$, then $\delta_{Y'}(G')\ge2$, because
(1) $\hat y$ has $x$ and a vertex of $X-\{x,x'\}$ (inherited from $y'$) as
neighbors, and (2) since $N_G(x')=\{y,y'\}$, all vertices of $Y-\{y,y'\}$ have
the same incident edges in $G'$ as in $G$.  Hence the induction hypothesis
applies to $G'$.

We next show $\Phi(G)\ge \Phi(G')+k-2$.  If an $X'$-matching in $G'$ uses an
edge $x\hat y$ in $G'$, then it extends to an $X$-matching in $G$ by using
$xy$ and $x'y'$.  If $xy'$ is an edge, then this matching also extends to
$k-1$ additional $X$-matchings in $G$ by using $xy'$ and copies of $x'y$.
Any $X'$-matching in $G'$ not using $x\hat y$ extends to $k-1$ $X$-matchings in
$G$ by adding copies of $x'y$.  If every $X'$-matching in $G'$ uses $x\hat y$
and $xy'$ is not an edge, then we still obtain $k-1$ $X$-matchings beyond those
using $xy$ and $x'y'$, since Case 2 for $G$ implies that every copy of $x'y$
lies in some $X$-matching in $G$.  Hence in all cases $\Phi(G)-\Phi(G')\ge k-2$.

If $\C X\ge3$, then $G'$ has at least $2k-2$ $X'$-matchings and
$\Phi(G)\ge 3k-4$, yielding $\Phi(G)\ge2k-1$ when $k=3$ and $\Phi(G)\ge2k$ when
$k\ge4$.

If we obtain only $2k-1$ $X$-matchings when $k=3$, then we must have only
$2k-2$ $X'$-matchings in $G'$, which when $k=3$ requires $\C{X'}=2$, so
$\C X=3$ and in fact $G'=F_3$.  To return from $G'$ to $G$, the merged vertex
$\hat y$ must be split into $y$ and $y'$.  Only one way to do this results in
every vertex of $X$ having degree at least $3$, and it produces $G_6$.
\end{proof}

The proof of Theorem~\ref{Y2+} suggests that there are various ways to 
have only $2k$ $X$-matchings under the conditions $\C X\ge3$ and $k\ge4$.

\section{Excess Vertices in $Y$}\label{sec:excess}

In this section we consider the effect of excess vertices in $Y$,
meaning we have an $X,Y$-bigraph satisfying Hall's Condition and 
$\C Y-\C X=t\ge1$.  We will prove Theorem~\ref{2kt} in various pieces.
In particular, we need to treat separately the cases $\C X=2$ and $\C X>2$.

\begin{construction}\label{X=2}
For $k-1>t\ge0$, define an $X,Y$-bigraph $L_{k,t}$ with $\C X=2$ and
$\C Y=t+2$ by adding to $K_{2,t+1}$ a single vertex $y^*$ forming $k-t-1$
edges with each vertex of $X$.  See Figure~\ref{tfig}.
Note that $L_{k,0}=F_k$ (Construction~\ref{sharp2}).

We have $\delta_X(L_{k,t})=k$, $\delta_Y(L_{k,t})=2$, and $\C Y-\C X=t$.
Each edge incident to $y^*$ can be extended to an $X$-matching in $t+1$ ways,
and there are $(t+1)t$ $X$-matchings that do not use $y^*$.  Hence
$\Phi(L_{k,t})=2(k-t-1)(t+1)+(t+1)t=(t+1)(2k-t-2)$.
\end{construction}

\begin{centering}
\begin{figure}[h]\label{tfig}
\gpic{
\expandafter\ifx\csname graph\endcsname\relax \csname newbox\endcsname\graph\fi
\expandafter\ifx\csname graphtemp\endcsname\relax \csname newdimen\endcsname\graphtemp\fi
\setbox\graph=\vtop{\vskip 0pt\hbox{%
    \graphtemp=.5ex\advance\graphtemp by 0.702in
    \rlap{\kern 0.117in\lower\graphtemp\hbox to 0pt{\hss $\bu$\hss}}%
    \graphtemp=.5ex\advance\graphtemp by 0.410in
    \rlap{\kern 0.117in\lower\graphtemp\hbox to 0pt{\hss $\bu$\hss}}%
    \graphtemp=.5ex\advance\graphtemp by 1.288in
    \rlap{\kern 0.995in\lower\graphtemp\hbox to 0pt{\hss $\bu$\hss}}%
    \graphtemp=.5ex\advance\graphtemp by 0.995in
    \rlap{\kern 0.995in\lower\graphtemp\hbox to 0pt{\hss $\bu$\hss}}%
    \graphtemp=.5ex\advance\graphtemp by 0.702in
    \rlap{\kern 0.995in\lower\graphtemp\hbox to 0pt{\hss $\bu$\hss}}%
    \graphtemp=.5ex\advance\graphtemp by 0.410in
    \rlap{\kern 0.995in\lower\graphtemp\hbox to 0pt{\hss $\bu$\hss}}%
    \graphtemp=.5ex\advance\graphtemp by 0.117in
    \rlap{\kern 0.995in\lower\graphtemp\hbox to 0pt{\hss $\bu$\hss}}%
    \special{pn 11}%
    \special{pa 117 702}%
    \special{pa 995 1288}%
    \special{fp}%
    \special{pa 995 1288}%
    \special{pa 117 410}%
    \special{fp}%
    \special{pa 117 410}%
    \special{pa 995 995}%
    \special{fp}%
    \special{pa 995 995}%
    \special{pa 117 702}%
    \special{fp}%
    \special{pa 117 702}%
    \special{pa 995 702}%
    \special{fp}%
    \special{pa 995 702}%
    \special{pa 117 410}%
    \special{fp}%
    \special{pa 117 410}%
    \special{pa 995 410}%
    \special{fp}%
    \special{pa 995 410}%
    \special{pa 117 702}%
    \special{fp}%
    \special{pn 8}%
    \special{ar -1040 -1985 2927 2927 0.801525 1.164062}%
    \special{ar 2153 2805 2927 2927 -2.340068 -1.977530}%
    \special{ar -121 -1771 2195 2195 1.036634 1.461458}%
    \special{ar 1235 2299 2195 2195 -2.104959 -1.680135}%
    \graphtemp=.5ex\advance\graphtemp by 0.702in
    \rlap{\kern 0.000in\lower\graphtemp\hbox to 0pt{\hss $x_2$\hss}}%
    \graphtemp=.5ex\advance\graphtemp by 0.410in
    \rlap{\kern 0.000in\lower\graphtemp\hbox to 0pt{\hss $x_1$\hss}}%
    \graphtemp=.5ex\advance\graphtemp by 1.288in
    \rlap{\kern 1.200in\lower\graphtemp\hbox to 0pt{\hss $y_{t+1}$\hss}}%
    \graphtemp=.5ex\advance\graphtemp by 0.410in
    \rlap{\kern 1.112in\lower\graphtemp\hbox to 0pt{\hss $y_1$\hss}}%
    \graphtemp=.5ex\advance\graphtemp by 0.117in
    \rlap{\kern 1.112in\lower\graphtemp\hbox to 0pt{\hss $y^*$\hss}}%
    \graphtemp=.5ex\advance\graphtemp by 1.580in
    \rlap{\kern 0.556in\lower\graphtemp\hbox to 0pt{\hss $L_{6,3}$\hss}}%
    \hbox{\vrule depth1.698in width0pt height 0pt}%
    \kern 1.200in
  }%
}%
}

\vspace{-1pc}
\caption{Construction~\ref{X=2}}
\end{figure}
\end{centering}

\begin{lemma}\label{X2t}
Let $G$ be an $X,Y$-bigraph with $\C X=2$ having a matching of size $2$.
If $\delta_X(G)\ge k\ge2$, $\delta_Y(G)\ge 2$, and $t=|Y|-|X|\ge 0$, then
$\Phi(G)\ge(t+1)(2k-t-2)$, which is sharp by $L_{k,t}$ in
Construction~\ref{X=2}.
\end{lemma}

\begin{proof}
The lower bound for $t=0$ is given in Theorem~\ref{Y2}.  Consider $t>0$.

Let $X=\{x_1,x_2\}$.  We split $Y$ into three sets:
$Y_0=N(x_1)\cap N(x_2)$, $Y_1=N(x_1)-N(x_2)$, and $Y_2=N(x_2)-N(x_1)$.
When we say ``with multiplicity'', we mean ``with multiplicity greater than
$1$''.

\smallskip
{\bf Case 1:} {\it $\C{Y_0}\ge 2$.}
We claim first that in some instance minimizing $\Phi(G)$ at most one
vertex of $Y_0$ has incident edges with multiplicity.  Suppose that
$x_1y$, $x_2y$, $x_1y'$ and $x_2y'$ have multiplicities $a,b,a',b'$,
respectively, where $y,y'\in Y_0$ with $y\ne y'$.  Let $c=d_G(x_1)$ and
$c'=d_G(x_2)$.  By symmetry, we have two cases: $a,a'\ge2$ or $a,b'\ge2$.

If $a,a'\ge2$, then each copy of $x_1y$ lies in $c'-b$ $X$-matchings, and each
copy of $x_1y'$ lies in $c'-b'$.  By symmetry, we may assume $b\ge b'$, and
then moving all but one copy of $x_1y'$ to become copies of $x_1y$ yields a
graph $G'$ having the desired properties and $\Phi(G')\le \Phi(G)$.

If we do not have $a,a'\ge2$ or $b,b'\ge2$, then with $a,b'\ge2$ we may also
assume $a'=b=1$.  Now each copy of $x_1y$ lies in $c'-1$ $X$-matchings, and the
copy of $x_1y'$ lies in $c'-b'$ $X$-matchings.  For each copy of $x_1y$ we
change into a copy of $x_1y'$, we reduce the number of $X$-matchings by $b'-1$.
Moving all but one copy of $x_1y'$ to become copies of $x_2y'$ yields a graph
$G'$ having the desired properties and $\Phi(G')< \Phi(G)$.

Hence we may assume that $y^*$ is the only vertex of $Y_0$ having edges with
multiplicity to $x_1$ or $x_2$; every other vertex of $Y_0$ has singleton edges
to both $x_1$ and $x_2$.  If $\C{Y_0}>1$, then we can choose a vertex
$y\in Y_0$ and let $G'=G-y$.  The induction hypothesis applies to $G'$ with
reduced values of $k$ and $t$, yielding 
$\Phi(G')\ge t(2(k-1)-(t-1)-2)$.  This count omits at least
$2(k-1)$ $X$-matchings in $G$ that use vertex $y$.
Hence $\Phi(G)\ge (t+1)2(k-1)-t(t+1)=(t+1)(2k-t-2)$, as desired.

\smallskip
{\bf Case 2:} {\it $\C{Y_0}=1$.}
Let $y^*$ be the unique vertex of $Y_0$.
If any vertex $y\in Y_1\cup Y_2$ has degree more than $2$ (by symmetry suppose
$y\in Y_1$), then we shift a copy of $x_1y$ to become a copy of $x_1y^*$
instead.  This loses $d_G(x_2)$ $X$-matchings and gains $d_G(x_2)-c$
$X$-matchings, where $c$ is the multiplicity of the edge $x_2y^*$.  Since this
reduces the number of $X$-matchings, we may assume that every vertex of
$Y_1\cup Y_2$ has degree exactly $2$ in $G$.

Let $m_1$ be the number of $X$-matchings not covering $y^*$, and let $m_2$ be
the number covering $y^*$.  Note that $\C{Y_1}+\C{Y_2}=t+1$, and each vertex in
$Y_1\cup Y_2$ is used in at least $2k$ matchings, since $\delta_Y(G)\ge2$ and
$\delta_X(G)\ge k$.  Summing this over all vertices of $Y_1\cup Y_2$ counts
each matching that avoids $y^*$ twice, so $2m_1+m_2\ge 2k(t+1)$.

Also, since vertices of $Y_1\cup Y_2$ have degree exactly $2$, we have
and $m_1=4|Y_1||Y_2|\le (t+1)^2$.
Now $\Phi(G)=m_1+m_2\ge 2k(t+1)-(t+1)^2=(t+1)(2k-t-1)>(t+1)(2k-t-2)$.

\smallskip
{\bf Case 3:} {\it $\C{Y_0}=0$.}
All $X$-matchings correspond to picking one edge at each of $x_1$ and $x_2$.
Letting $d_i=d_G(x_i)$, we have $\Phi(G)=d_1d_2$ and
$d(x_1)+d(x_2)\ge \max\{2k,2t+4\}$ with $d_1,d_2\ge k$.  To minimize $\Phi(G)$,
we let $\min\{d_1,d_2\}=k$.  This yields $\Phi(G)\ge \max\{k^2,k(2t+4-k)\}$.

If $k\ge t+2$, then we get $\Phi(G)\ge k^2$.  We need $k^2\ge (t+1)(2k-t-2)$.
Let $f(k)=k^2-(t+1)(2k-t-2)$.  Differentiating yields $f'(k)=2k-2(t+1)$, which
is positive when $k\ge t+1$, so it suffices to show $f(t+1)>0$.  We compute
$f(t+1)=(t+1)(t+1-2t-2+t+2)=t+1>0$.

If $k<t+2$, then we minimize $\Phi(G)$ by setting $d_1=k$ and $d_2=2t+4-k$.
Let $g(k)=k(2t+4-k)-(t+1)(2k-t-2)$.  Note that $g'(k)=-2k+2$, which is negative,
so it suffices to show $g(t+1)\ge0$.  We compute
$g(t+1)=(t+1)(2t+4-t-1-2t-2+t+2)=3(t+1)>0$.
\end{proof}

Having considered the case $\C X=2$, we now restrict to $\C X\ge3$.  In this
case we obtain the stronger bound $2k(t+1)$ in terms of $k$ and $t$ than the
$(2k-t-2)(t+1)$ of Lemma~\ref{X2t}.  We also consider relaxing
$\delta_Y(G)\ge2$ to the setting $\delta_Y(G)\ge1$, where we can only guarantee
$\Phi(G)\ge k(t+1)$.  We begin with sharpness constructions for both situations.

\begin{construction}\label{Gnkt}
As in Construction~\ref{sharp2}, let $H_{n,k}$ be the $X,Y$-bigraph with
$X=\{\VEC x1n\}$ and $Y=\{\VEC y1n\}$ whose edge set consists of one copy of
$x_iy_i$ for $1\le i\le n$ plus $k-1$ additional copies of $x_iy_1$ for
$1\le i\le n$ (so $x_1y_1$ has multiplicity $k$).

Fixing $k,t\ge1$ and $n\ge3$, we produce graphs $G_{n,k,t}$ and $G'_{n,k,t}$
from $H_{n-1,k}$.  Modify $H_{n-1,k}$ by adding one vertex $x^*$ to $X$ and
$t+1$ vertices to $Y$.  Make $x^*$ adjacent to all vertices in $Y$:
with multiplicity $1$ to produce $G_{n,k,t}$, with multiplicity $2$ to produce
$G'_{n,k,t}$.  If $n+t+1$ is small, add enough copies of $x^*y_1$ to increase
the degree of $x^*$ to $k$; these edges belong to no $X$-matching.
See Figure~\ref{tfigG}.  These graphs $G_{n,k,t}$ and $G'_{n,k,t}$ will be
sharpness examples for Theorems~\ref{nky1} and~\ref{nky2}, respectively.

Note that $\delta_X(G_{n,k,t})=\delta_X(G'_{n,k,t})=k$, and in each case
$\C Y-\C X=t$ and $\C X=n$.  Since $x_1$ is adjacent only to $y_1$, every
$X$-matching pairs them and then also pairs $x_i$ with $y_i$ for $2\le i\le n$.
Finally, $x^*$ is paired with one of the $t+1$ remaining vertices in $Y$.
With $\delta_Y(G_{n,k,t})=1$, we have $\Phi(G_{n,k,t})=k(t+1)$.
With $\delta_Y(G'_{n,k,t})=2$, we have $\Phi(G'_{n,k,t})=2k(t+1)$.

For $k=3$ and $\C X = \C Y-1=3$, obtain $G_7$ from a $6$-cycle by 
adding to $Y$ one vertex $y^*$ adjacent to each vertex of $X$, the result is a
simple $X,Y$-bigraph with $\delta_X=3$ and $\delta_Y=2$.  In $G_7$ there are
two $X$-matchings not using $y^*$.  Each edge $e$ incident to $y^*$ appears in
three $X$-matchings, since deleting the endpoints of $e$ leaves a $5$-vertex
path in which we only need to cover the second and fourth vertices to complete
an $X$-matching.  Hence $G_7$ has $11$ $X$-matchings; this will be the unique
exception to the lower bound $2k(t+1)$.
\end{construction}

\begin{centering}
\begin{figure}[h]\label{tfigG}
\gpic{
\expandafter\ifx\csname graph\endcsname\relax \csname newbox\endcsname\graph\fi
\expandafter\ifx\csname graphtemp\endcsname\relax \csname newdimen\endcsname\graphtemp\fi
\setbox\graph=\vtop{\vskip 0pt\hbox{%
    \graphtemp=.5ex\advance\graphtemp by 0.895in
    \rlap{\kern 0.184in\lower\graphtemp\hbox to 0pt{\hss $\bu$\hss}}%
    \graphtemp=.5ex\advance\graphtemp by 0.368in
    \rlap{\kern 0.184in\lower\graphtemp\hbox to 0pt{\hss $\bu$\hss}}%
    \graphtemp=.5ex\advance\graphtemp by 0.105in
    \rlap{\kern 0.184in\lower\graphtemp\hbox to 0pt{\hss $\bu$\hss}}%
    \graphtemp=.5ex\advance\graphtemp by 1.158in
    \rlap{\kern 0.974in\lower\graphtemp\hbox to 0pt{\hss $\bu$\hss}}%
    \graphtemp=.5ex\advance\graphtemp by 0.895in
    \rlap{\kern 0.974in\lower\graphtemp\hbox to 0pt{\hss $\bu$\hss}}%
    \graphtemp=.5ex\advance\graphtemp by 0.632in
    \rlap{\kern 0.974in\lower\graphtemp\hbox to 0pt{\hss $\bu$\hss}}%
    \graphtemp=.5ex\advance\graphtemp by 0.368in
    \rlap{\kern 0.974in\lower\graphtemp\hbox to 0pt{\hss $\bu$\hss}}%
    \graphtemp=.5ex\advance\graphtemp by 0.105in
    \rlap{\kern 0.974in\lower\graphtemp\hbox to 0pt{\hss $\bu$\hss}}%
    \special{pn 11}%
    \special{pa 184 895}%
    \special{pa 974 1158}%
    \special{fp}%
    \special{pa 184 895}%
    \special{pa 974 895}%
    \special{fp}%
    \special{pa 184 895}%
    \special{pa 974 632}%
    \special{fp}%
    \special{pa 184 895}%
    \special{pa 974 368}%
    \special{fp}%
    \special{pa 184 895}%
    \special{pa 974 105}%
    \special{fp}%
    \special{pa 184 368}%
    \special{pa 974 368}%
    \special{fp}%
    \special{pa 184 105}%
    \special{pa 974 105}%
    \special{fp}%
    \special{pn 8}%
    \special{ar -30 -1592 1974 1974 1.036634 1.461458}%
    \special{ar 1189 2067 1974 1974 -2.104959 -1.680135}%
    \special{ar 579 -1423 1579 1579 1.318116 1.823477}%
    \special{ar 579 1634 1579 1579 -1.823477 -1.318116}%
    \graphtemp=.5ex\advance\graphtemp by 0.895in
    \rlap{\kern 0.079in\lower\graphtemp\hbox to 0pt{\hss $x^*$\hss}}%
    \graphtemp=.5ex\advance\graphtemp by 0.368in
    \rlap{\kern 0.000in\lower\graphtemp\hbox to 0pt{\hss $x_{n-1}$\hss}}%
    \graphtemp=.5ex\advance\graphtemp by 0.105in
    \rlap{\kern 0.079in\lower\graphtemp\hbox to 0pt{\hss $x_1$\hss}}%
    \graphtemp=.5ex\advance\graphtemp by 0.368in
    \rlap{\kern 1.158in\lower\graphtemp\hbox to 0pt{\hss $y_{n-1}$\hss}}%
    \graphtemp=.5ex\advance\graphtemp by 0.105in
    \rlap{\kern 1.079in\lower\graphtemp\hbox to 0pt{\hss $y_1$\hss}}%
    \graphtemp=.5ex\advance\graphtemp by 0.895in
    \rlap{\kern 2.026in\lower\graphtemp\hbox to 0pt{\hss $\bu$\hss}}%
    \graphtemp=.5ex\advance\graphtemp by 0.368in
    \rlap{\kern 2.026in\lower\graphtemp\hbox to 0pt{\hss $\bu$\hss}}%
    \graphtemp=.5ex\advance\graphtemp by 0.105in
    \rlap{\kern 2.026in\lower\graphtemp\hbox to 0pt{\hss $\bu$\hss}}%
    \graphtemp=.5ex\advance\graphtemp by 1.158in
    \rlap{\kern 2.816in\lower\graphtemp\hbox to 0pt{\hss $\bu$\hss}}%
    \graphtemp=.5ex\advance\graphtemp by 0.895in
    \rlap{\kern 2.816in\lower\graphtemp\hbox to 0pt{\hss $\bu$\hss}}%
    \graphtemp=.5ex\advance\graphtemp by 0.632in
    \rlap{\kern 2.816in\lower\graphtemp\hbox to 0pt{\hss $\bu$\hss}}%
    \graphtemp=.5ex\advance\graphtemp by 0.368in
    \rlap{\kern 2.816in\lower\graphtemp\hbox to 0pt{\hss $\bu$\hss}}%
    \graphtemp=.5ex\advance\graphtemp by 0.105in
    \rlap{\kern 2.816in\lower\graphtemp\hbox to 0pt{\hss $\bu$\hss}}%
    \special{pn 11}%
    \special{pa 2026 368}%
    \special{pa 2816 368}%
    \special{fp}%
    \special{pa 2026 105}%
    \special{pa 2816 105}%
    \special{fp}%
    \special{pn 8}%
    \special{ar 1811 -1592 1974 1974 1.036634 1.461458}%
    \special{ar 3031 2067 1974 1974 -2.104959 -1.680135}%
    \special{ar 2421 -1423 1579 1579 1.318116 1.823477}%
    \special{ar 2421 1634 1579 1579 -1.823477 -1.318116}%
    \special{ar 3031 -803 1974 1974 1.680135 2.104959}%
    \special{ar 1811 2857 1974 1974 -1.461458 -1.036634}%
    \special{ar 2421 -633 1579 1579 1.318116 1.823477}%
    \special{ar 2421 2424 1579 1579 -1.823477 -1.318116}%
    \special{ar 1811 -1066 1974 1974 1.036634 1.461458}%
    \special{ar 3031 2593 1974 1974 -2.104959 -1.680135}%
    \special{ar 985 -1521 2632 2632 0.801525 1.164062}%
    \special{ar 3857 2785 2632 2632 -2.340068 -1.977530}%
    \special{ar -341 -2262 3947 3947 0.643501 0.927295}%
    \special{ar 5184 3263 3947 3947 -2.498092 -2.214297}%
    \graphtemp=.5ex\advance\graphtemp by 0.895in
    \rlap{\kern 1.921in\lower\graphtemp\hbox to 0pt{\hss $x^*$\hss}}%
    \graphtemp=.5ex\advance\graphtemp by 0.368in
    \rlap{\kern 1.842in\lower\graphtemp\hbox to 0pt{\hss $x_{n-1}$\hss}}%
    \graphtemp=.5ex\advance\graphtemp by 0.105in
    \rlap{\kern 1.921in\lower\graphtemp\hbox to 0pt{\hss $x_1$\hss}}%
    \graphtemp=.5ex\advance\graphtemp by 0.368in
    \rlap{\kern 3.000in\lower\graphtemp\hbox to 0pt{\hss $y_{n-1}$\hss}}%
    \graphtemp=.5ex\advance\graphtemp by 0.105in
    \rlap{\kern 2.921in\lower\graphtemp\hbox to 0pt{\hss $y_1$\hss}}%
    \graphtemp=.5ex\advance\graphtemp by 1.421in
    \rlap{\kern 0.579in\lower\graphtemp\hbox to 0pt{\hss $G_{3,3,2}$\hss}}%
    \graphtemp=.5ex\advance\graphtemp by 1.421in
    \rlap{\kern 2.421in\lower\graphtemp\hbox to 0pt{\hss $G'_{3,3,2}$\hss}}%
    \hbox{\vrule depth1.526in width0pt height 0pt}%
    \kern 3.000in
  }%
}%
}

\vspace{-1pc}
\caption{Construction~\ref{Gnkt}}
\end{figure}
\end{centering}

We will use Theorem~\ref{nky1} to prove Theorem~\ref{nky2}.

\begin{theorem}\label{nky1}
Let $G$ be an $X,Y$-bigraph with an $X$-matching, such that $|X|\ge 3$,
$\delta_X(G)\ge k\ge1$, and $\delta_Y(G)\ge 1$.  If $|Y|-|X|=t\ge 1$, then
$\Phi(G)\ge k(t+1)$, which is sharp by $G_{n,k,t}$ in Construction~\ref{Gnkt}.
\end{theorem}
\begin{proof}
We have the usual two cases depending on whether $G$ is $X$-surplus.

\smallskip
{\bf Case 1:} {\it $|N(S)|=|S|$ for some nonempty proper subset $S$ of $X$.}
By Theorem~\ref{main}, there are at least $k$ $S$-matchings in $G$.
Since $\Phi(G)\ge1$, for some $T\esub Y-N(S)$ there is a perfect matching in
the subgraph $G'$ induced by $(X-S)\cup T$.  Note that $N(T)=X-S$.
For each such matching $M$ and each vertex $y\in Y-N(S)-T$,
we can use an edge $xy$ with $x\in N(T)$ instead of the edge covering $x$ in
$M$ to obtain a matching $M'$ that combines with any $S$-matching.  Doing this
with each vertex of $Y-N(S)-T$ yields a total of $k(t+1)$ $X$-matchings
when $\delta_Y(G)\ge1$.

\smallskip
{\bf Case 2:} {\it $\C{N(S)}>\C S$ for every proper subset $S$ of $X$.}
We will use induction on $n+t$, where $n=\C X$.  We first prove the claim when
$t=1$.  Since $\C X\ge3$, we have $\C Y\ge 4$; consider $y_1,y_2,y_3,y_4\in Y$.
Obtain $G_1$ from $G$ by merging $y_1$ and $y_2$ into a new vertex $y'_1$; that
is, deleting $y_1$ and $y_2$ and let $y'_1$ inherit the incident edges from
both.  Obtain $G_2$ from $G$ by similarly merging $y_3$ and $y_4$ into a new
vertex $y'_2$.  Since neighborhoods of subsets in $X$ decrease by at most $1$
when moving from $G$ to $G_1$ or $G_2$, both satisfy Hall's Condition, and
Theorem~\ref{main} yields at least $k$ $X$-matchings in each.  Such matchings
also cover the modified $Y$.  Hence the $X$-matchings in $G_1$ come from
$X$-matchings in $G$ that cover one of $\{y_1,y_2\}$ and both of $\{y_3,y_4\}$,
while those in $G_2$ come from $X$-matchings in $G$ that cover one of
$\{y_3,y_4\}$ and both of $\{y_1,y_2\}$.  These sets of matchings are disjoint,
so $\Phi(G)\ge 2k$, as desired.

Thus we may assume $t>1$.  Let $Y'$ be the set of vertices in $Y$ with
exactly one neighbor.

\smallskip
{\bf Case 2a:} {\it There exist $y_1,y_2\in Y'$ having distinct neighbors.}
Let $x_1$ and $x_2$ be the neighbors of $y_1$ and $y_2$, respectively.
Obtain $G'$ from $G$ by merging $y_1$ and $y_2$ into a new vertex $y'$.
Since $G$ is near-surplus, $\Phi(G')\ge1$.  The induction hypothesis applies
to yield $kt$ $X$-matchings in $G'$, each of which comes from an $X$-matching
in $G$ covering at most one of $\{y_1,y_2\}$.  We need $k$ additional
$X$-matchings using both $y_1$ and $y_2$.  Let $G''=G-\{x_1,y_1,x_2,y_2\}-Y''$,
where $Y''$ consists of any other vertices in $Y'$ whose only neighbor lies
in $\{x_1,x_2\}$.  Since vertices of $X-\{x_1,x_2\}$ have no neighbors in $Y$
that were deleted, $G''$ satisfies Hall's Condition.  Although we no longer
can guarantee excess vertices in $G''$, still Theorem~\ref{main} guarantees
$k$ matchings in $G''$ the cover $X-\{x_1,x_2\}$, and adding $x_1y_1$ and
$x_2y_2$ to these yields the desired extra $k$ $X$-matchings in $G$.

\smallskip
{\bf Case 2b:} {\it $Y'\ne\nul$, and all vertices of $Y'$ have the same
neighbor $x'$ in $X$.}
Let $G'=G-x'-Y'$, and let $s=\C{Y'}$ and $X'=X-\{x'\}$.  Since
$N_{G'}(x)=N_G(x)$ for $x\in X'$, we can apply the induction hypothesis if
$\C X\ge 4$ to find $k(t-(s-1)+1)$ matchings in $G'$ that cover $X'$.  We can
extend each to an $X$-matching in $G$ in $s$ ways, and all these matchings are
distinct.  Since $s(t+2-s)$ is minimized on the allowed interval for $s$ when
$s=1$, we again obtain $\Phi(G)\ge k(t+1)$.

If $\C X=3$, then note that each vertex of $Y-Y'$ is adjacent to each vertex
of $X'$, since by the definition of $Y'$ the vertices of $Y-Y'$ have at least
two neighbors.  Hence when we choose an edge incident to some $x\in X'$, we can
complete an $X$-matching containing it in at least $s(t+2-s)$ ways.  This gives
at least $k(t+1)$ $X$-matchings starting with either vertex of $X'$, and we
may count each matching twice.  Again $\Phi(G)\ge k(t+1)$.

\smallskip
{\bf Case 2c:} {\it $Y'=\nul$, so $\C{N(y)}\ge 2$ for all $y\in Y$.}
Fix a vertex $x\in X$.  Among all $y\in N(x)$, let $y'$ be the neighbor $y$
that minimizes $\Phi(G-x-y)$ (this is nonzero for all $y$).  Obtain $G'$
from $G$ by replacing each edge incident to $x$ with a copy of $xy'$.
Note that $\Phi(G')\le \Phi(G)$.  Furthermore, each vertex of $Y$ has lost
at most one neighbor ($x$), so $\delta_Y(G')\ge1$, but vertices of $X-\{x\}$
have lost no neighbors.  Let $G''=G-x-y'$.  We can apply the induction
hypothesis to $G''$ to obtain $k(t+1)$ matchings covering $X-\{x\}$ in $G''$,
and each extends by copies of $xy'$ to $k$ $X$-matchings in $G'$.
Hence $\Phi(G)\ge k^2(t+1)$.
\end{proof}

\begin{theorem}\label{nky2}
Let $G$ be an $X,Y$-bigraph with an $X$-matching, such that $|X|\ge 3$,
$\delta_X(G)\ge k$, and $\delta_Y(G)\ge 2$.  If $|Y|-|X|=t\ge 1$, then
$\Phi(G)\ge 2k(t+1)$ (sharp by $G'_{n,k,t}$ in Construction~\ref{Gnkt}),
except that $\Phi(G)\ge 2k(t+1)-1$ when $(n,k,t)=(3,3,1)$ (sharp by $G_7$
in Construction~\ref{Gnkt}).
\end{theorem}

\begin{proof}
We have the usual two cases depending on whether $G$ is $X$-surplus.

{\bf Case 1:} {\it $|N(S)|=|S|$ for some nonempty proper subset $S$ of $X$.}
By Theorem~\ref{main}, there are at least $k$ $S$-matchings in $G$.
Let $X'=X-S$.  Since $G$ has an $X$-matching, there is a perfect matching in
the subgraph $G'$ induced by $X'\cup T$, for some $T\esub Y-N(S)$.  Note that
$N(T)=X'$.  Since $\delta_Y(G)\ge2$, applying Theorem~\ref{main} to $G'$ as a
$T,X'$-bigraph yields at least two perfect matchings in $G'$.  For each such
matching $M$ and each vertex $y\in Y-(N(S)\cup T)$, we can introduce an edge
$xy$ to replace the edge incident to $x$ in $M$.  Thus $\Phi(G)\ge k(2+2t)$.
(We cannot guarantee more here because $y$ may have only one neighbor in $X'$.)

{\bf Case 2:} {\it $|N(S)|>|S|$ for every nonempty proper subset $S$ of $X$.}
This case cannot occur unless $k\ge2$ holds, and it requires $\C{N(x)}\ge2$ for
all $x\in X$.  Note that after merging two vertices in $Y$, the resulting graph
still satisfies Hall's Condition for an $X$-matching; this allows us to use
induction on $t$.  For the base case, we show that $\Phi(G)\ge4k$ when $t=1$.
We will use Theorem~\ref{Y2}, which provides different lower bounds based on
$k$, so we consider three cases, for $k=2$, $k=3$, and $k\ge4$.

If $G$ is leafless, then $\C{E(G)}\ge 2\C Y=2(n+t)$, and hence $b\ge0$ in the
statement of Theorem~\ref{leafmain}, which yields
$\Phi(G)\geq [(n-1)\cdot 1+2](1+1)\geq 8$, since $n\ge3$.
Hence we may assume there exists $y\in Y$ with unique neighbor $x'$, with
$x'y$ having multiplicity at least $2$.
Let $G^*=G-y$ and $G'=G-y-x'$.  Let $X'=X-\{x'\}$ and $Y'=Y-y$.
Each $X$-matching in $G^*$ is also an $X$-matching in $G$.  Each $X'$-matching
in $G'$ extends to an $X$-matching in $G$ that is not in $G^*$ in at least two
ways by adding a copy of $x'y$.  Hence
$$\Phi(G)\geq \Phi(G^*)+2\Phi(G').$$
Since $G^*$ contains all edges incident to each vertex of $Y'$, we have
$\delta_{Y'}(G^*)\ge2$.  Hence Theorem~\ref{main} applies to $G^*$ as a
$Y',X$-bigraph to yield $\Phi(G^*)\geq 2$.  Also, by the characterization
of sharpness for $k=2$ in Theorem~\ref{Y2+}, $\Phi(G')\geq 3$. 
Thus, $\Phi(G)\geq 2+2\cdot 3=8$.

For $k\ge 3$, since $|X|\ge 3$ and $t\ge1$, we have distinct
$y_1,y_2,y_3,y_4\in Y$.  Merging two vertices in $Y$ yields a graph $G'$
satisfying Hall's Condition, so Theorem~\ref{Y2+} yields
$\Phi(G')\ge 2k$ when $k\ge4$ and $\Phi(G')\ge 2k-1$ when $k=3$.
These are perfect matchings in $G'$, since $t=1$.

These matchings when $y_1$ and $y_2$ are merged yield $X$-matchings in $G$ that
omit $y_1$ or $y_2$ and cover the rest of $Y$.  When $y_3$ and $y_4$ are
merged, the resulting $X$-matchings in $G$ omit $y_3$ or $y_4$ and cover the
rest of $Y$.  Hence in total we have $4k$ distinct $X$-matchings if $k\ge4$.

When $k=3$, Theorem~\ref{Y2+} gives us at least six $X$-matchings after each
merge unless $\C X=3=\C Y-1$ and the graph resulting from the merge is $G_6$.
Note that the merge does not lose any edges.  In $G_6$, one vertex of the
reduced part $Y'$ has degree $5$, while the other two vertices have degree $2$.
Hence $G$ cannot reduce to $G_6$ in both merges on a pairing of four vertices,
so $\Phi(G)\ge11$.  In order to have $\Phi(G)\ge11$, each pairing must produce
one copy of $G_6$.  Hence $G$ has one vertex in $Y$ with degree $3$ and three
vertices with degree $2$.  Using also that $G$ is $X$-surplus, we obtain
$\Phi(G)\ge12$ unless $G$ is $G_7$ as in Construction~\ref{Gnkt}.

Having completed the proof for $t=1$, we proceed by induction.
Suppose $t\ge2$.  Note that a merge of two vertices in $Y$ when $t=2$ cannot
produce $G_7$, because $G_7$ has maximum degree $3$, and merging any two
vertices with degree at least $2$ will produce a vertex with degree at least
$4$.  Hence when we perform a merge in $Y$, the resulting graph will satisfy
Hall's Condition and have at least $2kt$ $X$-matchings.
We consider three subcases.

\smallskip
{\bf Case 2a:} {\it There exist $y_1,y_2\in Y$ having distinct unique
neighbors.}  
Merging $y_1$ and $y_2$ yields $2kt$ $X$-matchings in $G$ that all use exactly
one of $y_1$ and $y_2$.  Since $\C X\ge 3$, when we delete $y_1$ and $y_2$
and their neighbors, we obtain an $X'',Y''$-bigraph $G''$ satisfying Hall's
Condition with $\delta_{X''}(G'')\ge k$.  Hence Theorem~\ref{main} applies
to yield at least $k$ $X''$-matchings in $G''$, each of which extends in at
least four ways to cover $y_1$ and $y_2$ and their neighbors.  Since these
$X$-matchings were not counted previously, $\Phi(G)\ge 2k(t+2)>2k(t+1)$.

\smallskip
{\bf Case 2b:} {\it Some vertex of $Y$ has only one neighbor, and all $y\in Y$
having only one neighbor have the same neighbor $x_1\in X$.}  Let
$Y_1=\{y\in Y\st\C{N(y)}=1\}$; note $Y_1\esub N(x_1)$. 

If $X=\{x_1,x_2,x_3\}$, then pick $y_1\in Y_1$ and $y_2\in Y-Y_1$.  Form
$G'$ from $G$ by merging $y_1$ and $y_2$.  By the induction hypothesis,
$G$ has at least $2kt$ $X$-matchings that cover at most one of $y_1$ and $y_2$.
If we can choose $y_2$ having exactly one neighbor in $\{x_2,x_3\}$, which
we may assume by symmetry is $x_2$, then we can match $x_1$ with $y_1$ in
two ways and match $x_2$ with $y_2$ in one way, without using any neighbor of
$x_3$.  Hence we can complete the $X$-matching by matching $x_3$ in at least
$k$ ways, yielding at least $2k$ $X$-matchings that cover both $y_1$ and $y_2$.

If no such vertex $y_2$ exists, then $x_2$ and $x_3$ are adjacent to all
of $Y-Y_1$.  Since $\C{Y-Y_1}=\C{N(\{x_2,x_3\})}>2$ in Case 2, we can
choose $y_3\in Y-Y_1-\{y_2\}$.  Merging $y_2$ and $y_3$ yields at least $2kt$
$X$-matchings in $G$ that cover at most one of $y_2$ and $y_3$.  Since both
$x_2$ and $x_3$ are adjacent to both of $y_2$ and $y_3$, we can match
$\{x_2,x_3\}$ into $\{y_2,y_3\}$ in at least two ways and match $x_1$ into
$Y_1$ in at least $k$ ways, again establishing the lower bound.

If $\C X\ge 4$, then $|X-\{x_1\}|\ge 3$.  Let $s=\C{Y_1}$.  Since
$\C{N(X-\{x_1\})}>n-1$, we have $1\le s\le t$.  There are at least $2s$ ways to
match $x_1$ into $Y_1$.  The graph $G-x_1$ satisfies the hypotheses of
Theorem~\ref{nky1}, since vertices of $Y-Y_1$ have a neighbor in $\{x_2,x_3\}$,
by the definition of $Y_1$.  Hence $\Phi(G-x_1)\ge k(t-s+2)$
and $\Phi(G)\ge 2k(t+2-s)s\ge 2k(t+1)$.

\smallskip
{\bf Case 2c:} {\it Every vertex of $Y$ has at least two neighbors.}
Since we are in the $X$-surplus case, every vertex of $X$ has at least two
neighbors, so $G$ is leafless.  Also $b\ge0$.  Hence Theorem~\ref{leafmain}
applies, and $\Phi(G)\ge[(n-1)t+2+b](t+1)$.  Since $\delta_X(G)\ge k$,
there are at least $nk$ edges in $G$, so $b\ge nk-2(n+t)$.
It suffices to show $(n-1)t+2+nk-2(n+t)\ge 2k$, or
$(n-3)t+2+n(k-2)\ge2k$.  Since $n\ge3$ by hypothesis, we need only show
$(n-2)k\ge 2(n-1)$, or $k/2\ge 1+1/(n-2)$.  Already $k\ge4$ suffices, and
we earlier completed the proof for $k=2$.  When $k=3$, we have the desired
inequality unless $n=3$.

In the case $k=n=3$, consider the needed inequality $(n-1)t+2+b\ge 2k$.
With $m$ being the number of edges, this reduces to
$2k\le (n-3)t+2+m-2n = m-4$.  This holds if $m\ge10$, so we may assume that
all three vertices of $X$ have degree exactly $3$.  With every vertex of $Y$
having at least two neighbors and only nine edges in total, this requires
$\C Y\le 4$.  Hence $t\le 1$, but we have already completed the case $t=1$.
\end{proof}

\vspace{-1pc}

\section*{Acknowledgement}

The authors thank J\'{o}zsef Balogh for a discussion that eventually leads to the Construction \ref{linearn}.

\end{document}